\documentclass[sn-basic, iicol]{sn-jnl}
\jyear{2021}%

\usepackage[utf8]{inputenc}
\usepackage{gensymb}

\theoremstyle{thmstyleone}%
\newtheorem{theorem}{Theorem}[section]

\newtheorem{lemma}[theorem]{Lemma}
\newtheorem{proposition}{Proposition}[section]

\DeclareMathOperator*{\argmin}{arg\,min}

\theoremstyle{thmstyletwo}%
\newtheorem{remark}{Remark}[section]

\theoremstyle{thmstylethree}%

\begin{document}

\title[ ]{Method of evolving junction on optimal path planning in flows fields}

\author*[1]{\fnm{Haoyan} \sur{Zhai}}\email{haoyanzhai@gmail.com}

\author[2]{\fnm{Mengxue} \sur{Hou}} \email{mhou30@gatech.edu}

\author[2]{\fnm{Fumin} \sur{Zhang}} \email{fzhang37@gatech.edu}

\author[1]{\fnm{Haomin} \sur{Zhou}}\email{hz25@gatech.edu}

\affil[1]{\orgdiv{Mathematics}, \orgname{Georgia Institute of Technology}, \orgaddress{\street{North Ave.}, \city{Atlanta}, \postcode{30332}, \state{GA}, \country{USA}}}

\affil[2]{\orgdiv{Electrical and Computer Engineering}, \orgname{Georgia Institute of Technology}, \orgaddress{\street{North Ave.}, \city{Atlanta}, \postcode{30332}, \state{GA}, \country{USA}}}


\maketitle

\abstract{We propose an algorithm using method of evolving junctions to solve the optimal path planning problems with piece-wise constant flow fields. In such flow fields, we prove that the optimal trajectories, with respect to a convex Lagrangian in the objective function, must be formed by piece-wise constant velocity motions. Taking advantage of this property, we transform the infinite dimensional optimal control problem into a finite dimensional optimization and use intermittent diffusion to solve the problems. The algorithm is proven to be complete. At last, we demonstrate the performance of the algorithm with various simulation examples.}

\keywords{optimal path planning, intermittent diffusion, method of evolving junctions.}


\section{Introduction}\label{intro}

Autonomous Underwater Vehicles (AUVs) are a class of submerged marine robots capable of performing persistent missions in the ocean. Over the last few decades, AUVs have been widely applied to various applications, including ocean sampling \cite{Leonard2010a,Smith2010}, surveillance and inspection \cite{Ozog2016LongtermMT}, and many more. Since most of the applications require the AUVs to execute long-term missions in unknown and dynamic oceanic environments with minimum human supervisions, their success is highly dependent on the level of autonomy of the AUVs. 

For robots operating in complex and dynamic environments, path planning is one of the crucial and fundamental functions to achieve autonomy. In short, the task is finding a feasible or optimal path, under the influence of a dynamic flow field, for an AUV to reach a predefined target point. Path planning has been studied extensively in robotics over the years. Several popular algorithms that have been applied to underwater vehicle navigation include graph based methods such as the A* method \cite{Rhoads2012,Pereira2013,kularatne2018going,kularatne2017optimal} and the Sliding Wavefront Expansion method \cite{Soulignac2011}, probability based methods like the Rapidly exploring Random Trees (RRTs) \cite{Lavalle1998,Kuffner2000,gammell2018informed,chen2019horizon,shome2020drrt}, and methods that approximate the solution of HJ (Hamilton-Jacobi) equations, such as the Level Set Method (LSM) \cite{Sethian1999,Lolla2016}. 

When the A* method is applied for an AUV, the continuous flow field is discretized into grid cells. At each step, it compares the cost of going from the current position to its neighboring cells so as to identify a path with the lowest cost. However, when the resolution, which is inverse proportional to the cell size, is not high enough, it may fail to find a feasible path even if there exists one. RRT and RRT* explore the flow field by using random sampling, with a bias towards the unexplored area \cite{karaman2011sampling,noreen2016optimal}. Like many other probabilistic based methods, RRTs provide, guaranteed in the asymptotic sense, a globally optimal solution only when the samples are dense enough. Methods to improve the convergence rate of RRT and RRT* have been introduced \cite{janson2015fast,gammell2018informed}. LSM computes a propagating front, incorporating both flow and vehicle speeds, to approximate the solution of the HJ equation. Then the optimal path is computed by back tracking along the normal direction of the wavefront from the destination position. LSM can plan a shortest time path over time-varying flow field, usually at the cost of longer computational time.

Using a regular grid to discretize the flow field can result in unnecessary large number of cells, which increases the computational burden of the planning methods. Since the flow speed in adjacent position is usually similar, the flow field can be partitioned into piece-wise constant subfields, within each the flow speed is a constant vector \cite{hou2019partitioning, kularatne2017optimal, kaiser2014cluster, ser2015flow}. 
Taking advantage of the flow field partitioning, we can parameterize a path by the index of cells crossed by the path, and the intersection points between the path and the boundaries of regions. In this way, the original path planning problem in the continuous flow field is reduced to a finite dimensional mixed integer optimization problem (MIP), with the decision variables being the index of cells crossed by the path, and the intersection points between the path and the boundaries of regions. Because of the coupling between the integer and continuous variable, the optimization problem is high dimensional, and is computationally expensive to solve. Existing works following this approach solves the MIP in a bi-level approach \cite{kularatne2017optimal,Soulignac2011}. The lower-level solver computes the minimum-cost continuous variable in each partitioned cell using convex optimization approaches, and then the upper-level solver optimizes the integer variable by the A* algorithm, with the branch cost of the decision tree being the optimal solution derived from the lower-level controller. However, because of the decoupling between the cell sequence and the intersection points, such algorithms cannot guarantee optimality of the solution, without imposing more assumptions on the problem \cite{sinha2017review}.  

In this paper, we present the interleaved branch-and-bound depth-first search with intermittent diffusion (iBnBDFS-ID) method for computing time and energy optimal paths for vehicles traveling in the partitioned flow field. Different from existing strategies that solves the MIP using a bi-level approach \cite{kularatne2018going,kularatne2017optimal,Soulignac2011}, we solve the MIP in an interleaved way, with guaranteed optimality. When a cell sequence connecting the start and the destination is found by the depth-first search (DFS), the intersection position between the optimal path and the cell boundary curves is computed by the intermittent diffusion (ID) method \cite{li2017method}. For large search tree with high branching factor, DFS has shown to be computationally impractical \cite{forrest1974practical}. To address this challenge, we propose a novel branch-and-bound (BnB) DFS method that reduces the number of nodes to be searched in the decision tree, and hence reduces the computational cost in solving the problem. The key insight is that we can find a lower bound of the the stage cost induced from traveling in one partitioned cell. Leveraging the lower bound, we compute a lower bound on the cost-of-arrival, and stop the DFS at the current node if lower bound on its cost-of-arrival exceeds the best solution found so far. 
Therefore, this work proposes a method to solve the AUV path planning problem, as an instance of MIP, without excessive demand on computing resources. Contributions of our work are as follows: 

\begin{enumerate}
    \item {\it Optimal solution structure}: We prove that for vehicle traveling in partitioned flow regions, the structure of the global optimal path is a piece-wise constant velocity motion, if the objective function is given as the total traveling time or the total energy consumption (modeled by a quadratic function). 
    \item {\it Interleaved solver for MIP}: We propose a novel computationally efficient method to solve the path planning method in the partitioned flow field, as an instance of the MIP, and proved completeness of the algorithm. Different from the existing works that decouples the process of optimizing the cell sequence assignment and the intersection position \cite{kularatne2017optimal,Soulignac2011}, our method optimizes the cell sequence assignment and the intersection position between the path and the cell boundaries at the same time. By incorporating a BnB technique in the DFS, the method avoids exhaustive search through the full decision tree. Hence it solves the planning problem with low computation cost.
    
     \end{enumerate}
   We evaluate the performance of the proposed algorithm through simulation of AUV path planning in both simulated and realistic ocean flow fields, and compare its computational cost and accuracy with the LSM method. The proposed method achieves comparable path cost, and has significantly lower computation cost compared with the LSM method.

In the next section, we present the formulation of the problem in the optimal control framework and the assumptions used in the paper. In Section \ref{algorithm}, we show how to transform the original problem into the finite dimensional optimization, and provide the algorithm. In Section \ref{completeness}, proof of the completeness of the algorithm is offered, accompanied with several numerical experiments in Section \ref{experiments}. At last, we end our paper with a brief conclusion.

\section{Problem Statement}\label{problem_statement}
We consider the vehicle moving in the space $\Omega\subset\mathbb{R}^d$ with dimension $d$, equipped with a dynamic $\dot{x}=u+v$, where $u$ is the environment flow velocity, and $v$ is the vehicle velocity or the control variable, satisfying $v\in\mathcal{U}$, providing that $\mathcal{U}$ is a compact space such that $\|v\|\leq V$, where $V$ is the max speed of the vehicle. We further assume that by taking advantage of existing methods on ocean flow field partitioning \cite{hou2019partitioning,10.3389/frobt.2021.575267}, the flow field is divided into finite number of convex regions $\{R_\alpha\}_{\alpha\in I_R}$ by boundary curves (surfaces if is in $\mathbb{R}^d,d\geq3$) $\{f_{\alpha\beta}\}_{(\alpha,\beta)\in I}$. Here
\begin{equation*}
I=\{(\alpha,\beta)\in I_R\times I_R:dim(\partial R_\alpha\cap\partial R_\beta)=d-1\},
\end{equation*}
where $dim(S)$ returns the dimension of the set $S$ and $\partial S$ is the boundary of $S$, and $f_{\alpha\beta}(x)=0$ is the $(d-1)$-dimensional compact boundary of the region $R_\alpha$ and $R_\beta$ (can be denoted as $\partial R_\alpha$ and $\partial R_\beta$ respectively). Let $x$ denote a point on the boundary, we can parameterize the cell boundary by defining a piece-wise diffeomorphism
\begin{equation*}
x(\lambda):D\subset\mathbb{R}^{d-1}\longrightarrow\left\lbrace y:f_{\alpha\beta}(y)=0\right\rbrace,
\end{equation*}
where $D$ is a $(d-1)$-dimensional unit ball.
Also within each region, we suppose that the flow velocity $u$ is a constant vector. Hence we can denote the flow velocity in each region $R_\alpha$ separately by $u_\alpha$. The vehicle needs to be controlled from an initial position $x_0$, crossing different regions and finally reaching the target position $x_f$.



Since there could be infinitely many feasible paths linking $x_0$ and $x_f$, a cost function is introduced to measure the travel expense with respect to different potential trajectories. We denote the cost function to be
\begin{equation}\label{cost_func}
J(v,T)=\int_0^TL(v(t))dt,
\end{equation}
letting $\gamma(t)$ be a continuous path with $\gamma(0)=x_0$, $\gamma(T)=x_f$ and $\dot{\gamma}(t)=u+v$. In this paper, we discuss the problem with the cost function specifically being the total travel time, that is $L(x,v)=1$ and the kinetic energy combining a constant running cost, in which case $L(x,v)=\|v\|^2+C$ where $C$ and $\|v\|$ are not simultaneously $0$ for all $x\in\Omega$ because of technical issues which will be discussed in Section \ref{completeness_energy}. Our goal is to find the optimal control function $v(t)$ with minimum cost, and can be expressed as the following problem:
\begin{equation}\label{problem}
\begin{aligned}
\min_{v,T}&\int_0^TL(v)dt\\
s.t.\ \ \ &\dot{x}=v+u,\\
&x(0)=x_0,\\
&x(T)=x_f,\\
&\max_{t\in[0,T]}\|v\|\leq V
\end{aligned}
\end{equation}

In the next section, we will concretely discuss our method, the idea of which is to seek a way to change the infinite dimensional optimal control problem into a finite dimensional optimization problem.

\section{Our Method}\label{algorithm}

All possible trajectories will continuously pass through a sequence of regions. 
Since the flow velocity $u$ is a constant vector in each region, we have the following theorem that will be proved in the Appendix:
\begin{theorem}\label{thm:general}
In the optimal solution for 
\begin{align*}
\min_{v,T}&\int_0^TL(v)dt\\
s.t.\ \ \ &\dot{x}=f(v+u),\nonumber\\
&x(0)=x_0,\nonumber\\
&x(T)=x_f,\nonumber\\
&\max_{t\in[0,T]}\|v\|\leq V\nonumber.
\end{align*}
where $u$ a constant vector and $f$ is invertible, the velocity $v$ is a constant vector.
\end{theorem}

With Bellman's principle, it is the fact that the optimal path should be formed via a piece-wise constant velocity motion. Thus, we restrict the velocity of the vehicle in each region to be a constant. Further, the regions are convex, therefore, the straight line path always lies inside each of the cell crossed and the path is continuous. Now we can introduce the notation $v_i$ to be the vehicle velocity in the $i^{th}$ cell crossed.

Because of Theorem \eqref{thm:general}, we can parameterize solution to \eqref{problem} by i) the cell sequence that the path goes through, and ii) the position where the path intersects with the cell boundaries. We assume that the initial and the terminal position in \eqref{problem} are on the boundary of the partitioned cells. In general this can be achieved by incorporating a fixed rule to partition the cells containing $x_0$ and $x_f$ into two separate cells, with the new generated boundary going through $x_0$ and $x_f$. 

Let $\mathcal{C} = \{c_1, c_2, \dots, c_n\}$ denote a sequence of the index of cells that a path travels through, $c_i\in I_R, \forall i\in[1,n]$.
We define  the junction set $\{x_i\}_{i=0}^n$ to be the intersections between a trajectory and the cell boundaries. The junction $x_i$ is on the boundary between the two regions indexed as $c_{i}$ and $c_{i+1}$.  
The junction $x_0$ is the initial junction and the destination $x_f$ is the last junction e.g.  $x_n=x_f$. 
For each region $R_i$, the entrance junction is $x_{i-1} $ and the exit junction is $x_i$. Fig.~\ref{fig:example_workspace} shows illustration of parameterizing a feasible path using junctions and the cell sequence. 

\begin{figure}
    \centering
    \includegraphics[width=0.4\textwidth]{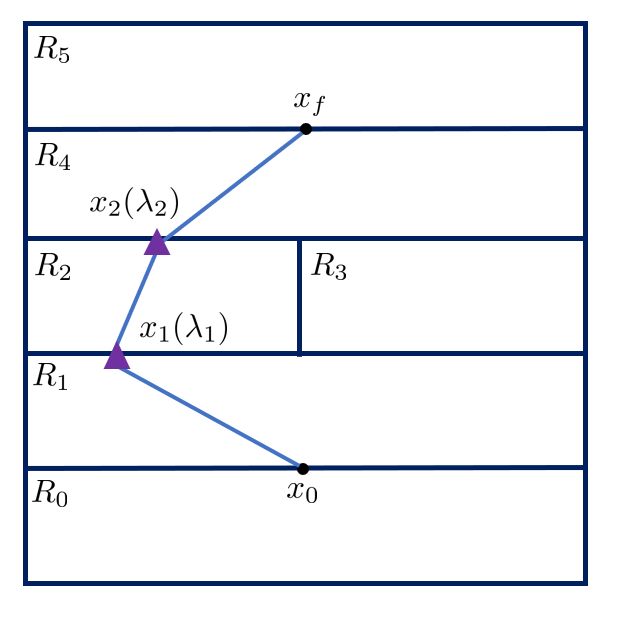}
    \caption{Example of parameterizing a feasible path in the partitioned domain by the cell sequence and the junction position. The purple triangles are the junction position. The path goes through 3 cells. Index of 3 cells are $c_1 = 1, c_2 = 2, c_3 = 4$. On the boundary $f_{12}$ and $f_{24}$ there are two junctions, position of which is parameterized by $\lambda_1$ and $\lambda_2$. }
    \label{fig:example_workspace}
\end{figure}

Given the above mentioned parameterization to feasible trajectories, to solve the planning problem, we need to i) determine the sequence of regions that the optimal path goes through, and ii) compute the optimal position of intersection position between the trajectory and the cell boundaries. 
In this section, we propose a novel method to solve the planning problem. The key idea of the method is based on two insights. First,  we only need to determine the cell sequence when the vehicle reaches cell boundary. Each assignment of visited cell sequence will produce a branch of a decision tree, with each branch in the tree representing a cell that the vehicle crosses. The decision tree contains a finite number of branches, and a feasible solution must be a path from the root of the tree to the target position. Second, the optimal junction position can be easily computed if we fix the assignment of cell sequence.

\subsection{Minimize total travel time}\label{min_travel_time}
When minimizing total travel time is the goal, we further let the vehicle move in maximum speed $V$. Let $u_{c_i}$ denote the flow speed in the cell indexed by $c_i$. Then the cost function can be converted in the format below:
\begin{equation*}
J(v,T, \mathcal{C})=\int_0^Tds=T=\sum_{i=1}^n g^t_i(x_i, x_{i-1}, c_i),
\end{equation*}
where $g^t_i$ is the travel time in region $R_{c_i}$ and is expressed as:
\begin{equation*}
g^t_i=\frac{\|x_i-x_{i-1}\|}{\|v_i+u_{c_i}\|}.
\end{equation*}
Since flow and vehicle velocities are constant in each cell, the motion must be in straight line, which leads to
\begin{align*}
&\frac{x_i-x_{i-1}}{\|x_i-x_{i-1}\|}=\frac{v_i+u_{c_i}}{\|v_i+u_{c_i}\|}\\
\Longrightarrow&\|v_i+u_{c_i}\|^2-\frac{2(x_i-x_{i-1})^T u_{c_i}}{\|v_i+u_{c_i}\|}+\|u_{c_i}\|^2-V^2=0\\
\Longrightarrow&\|v_i+u_{c_i}\|=\frac{(x_i-x_{i-1})^Tu_{c_i}}{\|x_i-x_{i-1}\|} \\
&\pm\left(\left(\frac{(x_i-x_{i-1})^Tu_{c_i}}{\|x_i-x_{i-1}\|}\right)^2+V^2-\|u_{c_i}\|^2\right)^{\frac{1}{2}}.
\end{align*}
To minimize the travel time, we take the plus sign so that
\begin{equation}\label{v_plus_u_norm}
\begin{aligned}
&\|v_i+u_{c_i}\| =\frac{(x_i-x_{i-1})^Tu_{c_i}}{\|x_i-x_{i-1}\|}+\\
&\left(\left(\frac{(x_i-x_{i-1})^Tu_{c_i}}{\|x_i-x_{i-1}\|}\right)^2+V^2-\|u_{c_i}\|^2\right)^{\frac{1}{2}}
\end{aligned}
\end{equation}
and have the travel time in region $R_{c_i}$ as
\begin{align*}
g_i^t&=\frac{1}{\|u_{c_i}\|^2-V^2}\Big((x_i-x_{i-1})^Tu_{c_i}-\\
&\sqrt{\left((x_i-x_{i-1})^Tu_{c_i}\right)^2+\|x_i-x_{i-1}\|^2(V^2-\|u_{c_i}\|^2)}\Big).
\end{align*}
At last, since $x_i$ is on the boundary of the regions indexed by $c_{i}$ and $c_{i+1}$, we have the smooth parameterization of each $x_i$ as $x_i=x_i(\lambda_i)$ where $\lambda_i\in D\subset\mathbb{R}^{d-1}$, transforming finally the cost function to be
\begin{equation*}
\begin{aligned}
J(\lambda_1,\cdots,\lambda_{n-1}, \mathcal{C})=&g^t_1(x_1(\lambda_1),x_0, c_1)\\
+&g^t_n(x_{n-1}(\lambda_{n-1}),x_f, c_n)\\+&\sum_{i=2}^{n-1}g^t_i(x_i(\lambda_i),x_{i-1}(\lambda_{i-1}), c_{i}).
\end{aligned}
\end{equation*}
and the problem is changed to a finite dimensional optimization problem
\begin{equation}\label{time_func}
\min_{\lambda_i\in D, c_i\in I_R, i=1,\cdots,n-1}J(\lambda_1,\cdots,\lambda_{n-1}, \mathcal{C}).
\end{equation}


\subsection{Minimize energy}\label{min_energy}
If the cost function is the kinetic energy with a constant running cost $C\geq0$
\begin{equation*}
J(v,T)=\int_0^T\|v\|^2+Cdt,
\end{equation*}
we first consider the constant velocity motion in each region $R_i$ within time $t_i$. Fixing the entrance and exit junctions $x_{i-1}$ and $x_i$, we can derive vehicle speed as
\begin{equation*}
v_i=\frac{x_i-x_{i-1}}{t_i}-u_{c_i},
\end{equation*}
and 
\begin{equation}\label{energy_speed_time_condition}
\|v_i\|^2=\frac{\|x_i-x_{i-1}\|^2}{t_i^2}+\|u_{c_i}\|^2\-\frac{2(x_i-x_{i-1})^Tu_{c_i}}{t_i}.
\end{equation}
the cost function in the specific region $R_{c_i}$ can be rewritten as
\begin{equation}\label{lb_energy}
\begin{aligned}
J(t_i)&=\int_0^{t_i}\|v_i\|^2+Cds\\
&=\frac{\|x_i-x_{i-1}\|^2}{t_i}+(\|u_{c_i}\|^2+C)t_i\\
&-2(x_i-x_{i-1})^Tu_{c_i}\\
&\geq2\sqrt{\|u_{c_i}\|^2+C}\|x_i-x_{i-1}\| \\
& -2(x_i-x_{i-1})^T u_{c_i}
\end{aligned}
\end{equation}
with equality holds at
\begin{equation}\label{energy_opt_time}
t_i^*=\frac{\|x_i-x_{i-1}\|}{\sqrt{\|u_i\|^2+C}}.
\end{equation}
Hence, if the maximum vehicle speed is large enough, that is, 
\begin{equation}\label{equ:energy_equation}
    V^2 \geq \frac{\|x_i-x_{i-1}\|^2}{t_i^{*2}}+\|u_{c_i}\|^2-\frac{2(x_i-x_{i-1})^T u_{c_i}}{t_i^*},
\end{equation}
we can find the optimal vehicle forward speed by replacing $t_i$ in \eqref{energy_speed_time_condition} with the optimal solution in \eqref{energy_opt_time}:
\begin{equation}\label{energy_speed_condition}
    \|v_i\|^2=2\|u_{c_i}\|^2+C-\frac{2(x_i-x_{i-1})^Tu_{c_i}}{\|x_i-x_{i-1}\|}\sqrt{\|u_{c_i}\|^2+C}.
\end{equation}
Hence, if \eqref{equ:energy_equation} holds,
we let the vehicle move in the speed of $\|v_i\|$ in \eqref{energy_speed_condition}. 

However, if the \eqref{equ:energy_equation} does not hold, we set the vehicle speed to be the maximum speed $V$ and the cost function is reduced to be $J=(V^2+C)t_i^*$ where
\begin{equation}\label{energy_t_i_star}
\begin{aligned}
t_i^*=&\frac{1}{\|u_{c_i}\|^2-V^2}\Bigg((x_i-x_{i-1})^Tu_{c_i}-\\&\Big(\left((x_i-x_{i-1})^Tu_{c_i}\right)^2\\
+& \|x_i-x_{i-1}\|^2(V^2-\|u_{c_i}\|^2)\Big)^{\frac{1}{2}}\Bigg),
\end{aligned}
\end{equation}
which is a root of
\begin{equation*}
(\|u_{c_i}\|^2-V^2)t_i^2-2(x_i-x_{i-1})^Tu_{c_i}t_i+\|x_i-x_{i-1}\|^2=0.
\end{equation*}
Thus the energy spent in each region is
\begin{align*}
&g_i^e(x_i,x_{i-1})=\\
&\left\lbrace\begin{array}{ll}
    \begin{aligned}
    &2\sqrt{\|u_{c_i}\|^2+C}\|x_i-x_{i-1}\|\\
    &-2(x_i-x_{i-1})^Tu_i
    \end{aligned}& \text{if \eqref{equ:energy_equation} holds} \\
    (V^2+C)t_i^* & \text{otherwise}
\end{array}\right.
\end{align*}
where $t_i^*$ is defined in \eqref{energy_t_i_star}. By using the same parametrization as in the time-optimal planning, we finally have the problem to be a finite dimensional optimization formulated as
\begin{equation}\label{energy_func}
\min_{\lambda_i\in D,c_i\in I_R, i=1,\cdots,n-1}J(\lambda_1,\cdots,\lambda_{n-1}, \mathcal{C})
\end{equation}
where 
\begin{equation*}
\begin{aligned}
J(\lambda_1,\cdots,\lambda_{n-1}, \mathcal{C})=&g^e_1(x_1(\lambda_1),x_0, c_1)\\
+&g^e_n(x_{n-1}(\lambda_{n-1}),x_f, c_n)\\+&\sum_{i=2}^{n-1}g^e_i(x_i(\lambda_i),x_{i-1}(\lambda_{i-1}), c_i).
\end{aligned}
\end{equation*}

Furthermore, $g^t_i$ and $g^e_i$ has the following properties: 
\begin{proposition}\label{differentiable_prop_energy}
If there exists a feasible trajectory from $x_{i-1}$ to $x_i$ in $R_i$ and $V\neq\|u_i\|$, then $g_i^t(x_i,x_{i-1}, c_i)$ and $g_i^e(x_i,x_{i-1}, c_i)$ are differentiable.
\end{proposition}
We give the proof of this property in the Appendix. With this proposition, we can take the derivative of the objective function, which is pivotal for applying the Intermittent Diffusion method described in Section \ref{subsec:intermittent diffusion}.



\subsection{Construction of decision tree}
We introduce a tree structured graph to model how the sequence of cell that the path crosses affects the total cost. In the tree, each node represents a boundary curve that contains a junction point. 
To construct the decision tree, we first define two boundary curves $f_{c_1, c_2}, f_{c_3, c_4}$ as adjacent if 
\begin{equation}\label{adjacent_condition}
   \{c_1, c_2\}\cap\{c_3, c_4\} \neq \emptyset,
\end{equation}
indicating that the two curves are two boundaries of the same cell.

\subsubsection{Decision tree traversal}
Starting from the boundary curve containing the initial position, which is the root node, a directed decision tree can be formed by connecting the adjacent boundary curves in the domain. The branch generation stops when the node is the boundary curve containing the destination position. Each node in the decision tree represents a boundary curve that a feasible path will go through. A connected path in the decision tree, starting from the root node to the target node represent a sequence of cells that a feasible path will cross. Fig.~\ref{fig:tree_struct} shows the decision tree constructed from a partitioned workspace. 

The construction of the entire decision tree is not necessary, and is time-consuming. However, for the purpose of clearly presenting the concept of the planning method, we will discuss how the tree can be fully constructed, and then present the branch pruning technique. 
We construct the decision tree iteratively using the depth-first search method.   Let $n_c$ represent the current node. At each step, we search for all the boundary curves adjacent to the current node, and call the current node as the predecessor of the new node. Since the optimal path will not visit one boundary curve more than one time, the optimal cell sequence should not include loops. Hence, when searching for the adjacent nodes to be expanded next, we will not expand an adjacent node of $n_c$ if it is already visited. This node generation process terminates when the target node is visited. 

By constructing the decision tree we can find all the cell sequences connecting the root and the terminal node. For one cell sequence, the MIP \eqref{time_func} and \eqref{energy_func} reduces to a finite dimensional non-convex optimization problem over the junction positions. Next we show how this optimization problem can be solved using the Intermittent Diffusion method.

\subsubsection{Intermittent Diffusion} \label{subsec:intermittent diffusion}
The objective functions in \eqref{time_func} and \eqref{energy_func} are both differentiable. Hence we use the Intermittent Diffusion (ID) to get the global minimizer \cite{chow2013global}, the key idea of which is adding white noise to the gradient flow intermittently. Namely, we solve the following stochastic differential equation (SDE) on the configuration space
\begin{equation}\label{id_sde}
d\lambda=-\nabla J(\lambda)d\theta+\sigma(\theta)dW(\theta)
\end{equation}
where $\lambda=(\lambda_1,\cdots,\lambda_{n-1})\in D^{n-1}$ and $W(\theta)$ is the standard Brownian motion. The diffusion is a piece-wise constant function defined by
\begin{equation*}
\sigma(\theta)=\sum_{i=1}^N\sigma_iI_{[S_i,T_i]}(\theta)
\end{equation*}
where $\sigma_i$ are constant and $I_{[S_i,T_i]}(\theta)$ is the characteristic function on interval $[S_i,T_i]$ with $0\leq S_1<T_1<S_2<T_2<\cdots<S_N<T_N<S_{N+1}=T$.

Thus, if $\sigma(\theta)=0$, we obtain the gradient flow back while when $\sigma(\theta)\neq0$, the solution of \eqref{id_sde} has positive probability to escape the current local minimizer. The theory of ID indicates that the solution of \eqref{id_sde} visits the global minimizer of $J$ with probability arbitrarily close to $1$ if $\min_i\lvert T_i-S_i \rvert$ is large enough, which is guaranteed by Theorem \ref{thm:ID} in Appendix. 

We use the forward Euler discretization to discretize the above SDE and get
\begin{equation}\label{update}
\lambda^{k+1}=\lambda^k-h\nabla J(\lambda^k)+\sigma_k\xi^k\sqrt{h}.
\end{equation}
The constant $h$ is the step size, $\sigma_k$ is the coefficient chosen to add the intermittent perturbation and $\xi^k\sim\mathcal{N}(0,1)$ is a Gaussian random variable. In practice, the global minimizer can be reached by tuning the white noise strength $\sigma_k$ as well as setting the total evolution round $N$ long enough.

We summarize the ID algorithm in Algorithm \ref{alg_2} with the objective function being \eqref{time_func} or \eqref{energy_func}.\\

\subsubsection{Branch cost lower bound}
We leverage a BnB technique to prune the branches in the decision tree, in order to save the computation cost in both node generation and solving for the optimal junction position. We approach this problem by leveraging the lower bound of the decision tree branch cost. In the decision tree, each branch represents a path segment connecting the junction on one boundary to the junction on one of its adjacent boundary. Hence, the branch cost of the decision tree represents the stage cost generated from going from one junction to another. However, since the optimal junction position is unknown when we construct the decision tree, the exact optimal branch cost cannot be calculated. Hence we introduce the lower bound of the branch cost, and use the DFSBnB technique \cite{poole2010artificial} to prune some of the branches in the tree. For the time-optimal and energy-optimal planning, we find the upper bound of the branch cost as follows. Let us define the maximum and minimum distance between two junctions $x_{i-1}$ and $x_i$,
\begin{equation*}
\begin{aligned}
    d_{\rm max} &= \max\limits_{x_i, x_{i-1}} \|x_i - x_{i-1}\|, \\
    \text{s.t.} \,&f_{c_i,c_{i+1}}(x_i) = 0, f_{c_{i-1}, c_i}(x_{i-1})=0, \\
    d_{\rm min} &= \min\limits_{x_i, x_{i-1}} \|x_i - x_{i-1}\|, \\
    \text{s.t.} \,&f_{c_i,c_{i+1}}(x_i) = 0, f_{c_{i-1}, c_i}(x_{i-1})=0.
\end{aligned}
\end{equation*}
We can find a lower bound on the minimum time spent on the path segment,
\begin{equation}\label{travel_time_minimum}
\begin{aligned}
    g_i^t & \geq \frac{1}{V^2 - \|u_{c_i}\|^2} ( \|x_i -x_{i-1}\| V \\
    & - \sqrt{\|x_{i} - x_{i-1}\|^2\|u_{c_i}\|^2 - ((x_i - x_{i-1})^T u_{c_i})^2)} \\
    & - (x_i-x_{i-1})^T u_{c_i})\\
    & \geq \frac{1}{V^2 - \|u_{c_i}\|^2}(d_{\rm min} V - \|(x_i - x_{i-1})\|\|u_{c_i}\| \\
    & + \|(x_i - x_{i-1})^T u_{c_i}\| - (x_i- x_{i-1})^T u_{c_i})\\
    & \geq \frac{d_{\rm min}}{ V + \|u_{c_i}\|}\triangleq g_{i, lb}^t.
\end{aligned}
\end{equation}
This lower bound is the travel time when the vehicle travels in the largest possible total speed $V + \|u_{c_i}\|$, which is the situation that $u_{c_i}$ is in the same direction as the shortest path segment from $x_{i-1}$ to $x_i$.

Similarly, given \eqref{lb_energy} we find a lower bound, denoted as $g_{i, lb}^e$ on the energy spent in one cell,
\begin{equation}\label{energy_minimum}
    g_{i,lb}^e = \max\{2d_{\rm min}\sqrt{\|u_{c_i}^2 + C\|}  - 2d_{\rm max} \|u_{c_i}\|,  0\}.
\end{equation}

\subsubsection{Branch-and-Bound method}
We use the DFS algorithm to iteratively generate the nodes, starting from the root and terminates when it reaches the destination node. After the first cell sequence connecting the root with the destination node is found, we use ID algorithm to compute the optimal junctions that result in the minimum total travel cost (Line \ref{line12}). To avoid traversing all feasible cell sequences, the algorithm maintains the lowest-cost path to the target found so far, and its cost. At each step of node generation, for each of the adjacent cell $m_i$ of the current node $n_c$, we compute a lower bound of the total cost of arrival, from the root to $v_i$, 
\begin{equation}\label{lb_coa}
    f_{g}(m_i) = f_g(n_c) + g_{i, lb},
\end{equation}
where for the time-optimal planning, $g_{i, lb} = g_{i, lb}^t$ is computed by \eqref{travel_time_minimum}, and for the energy-optimal planning, $g_{i, lb} = g_{i, lb}^e$ is computed by \eqref{energy_minimum}. If $f_g(m_i)$ is larger than the lowest-cost path found so far, then all the path that goes through the cell boundary represented by $m_i$ cannot be the optimal solution, since its total cost will be larger than the lowest-cost path found so far. Thus we stop the DFS from $m_i$ to its child nodes, and go to the next adjacent cell of $n_c$ to continue the search (Line \ref{line24}).

\begin{figure*}
    \centering
    \includegraphics[width=0.8\textwidth]{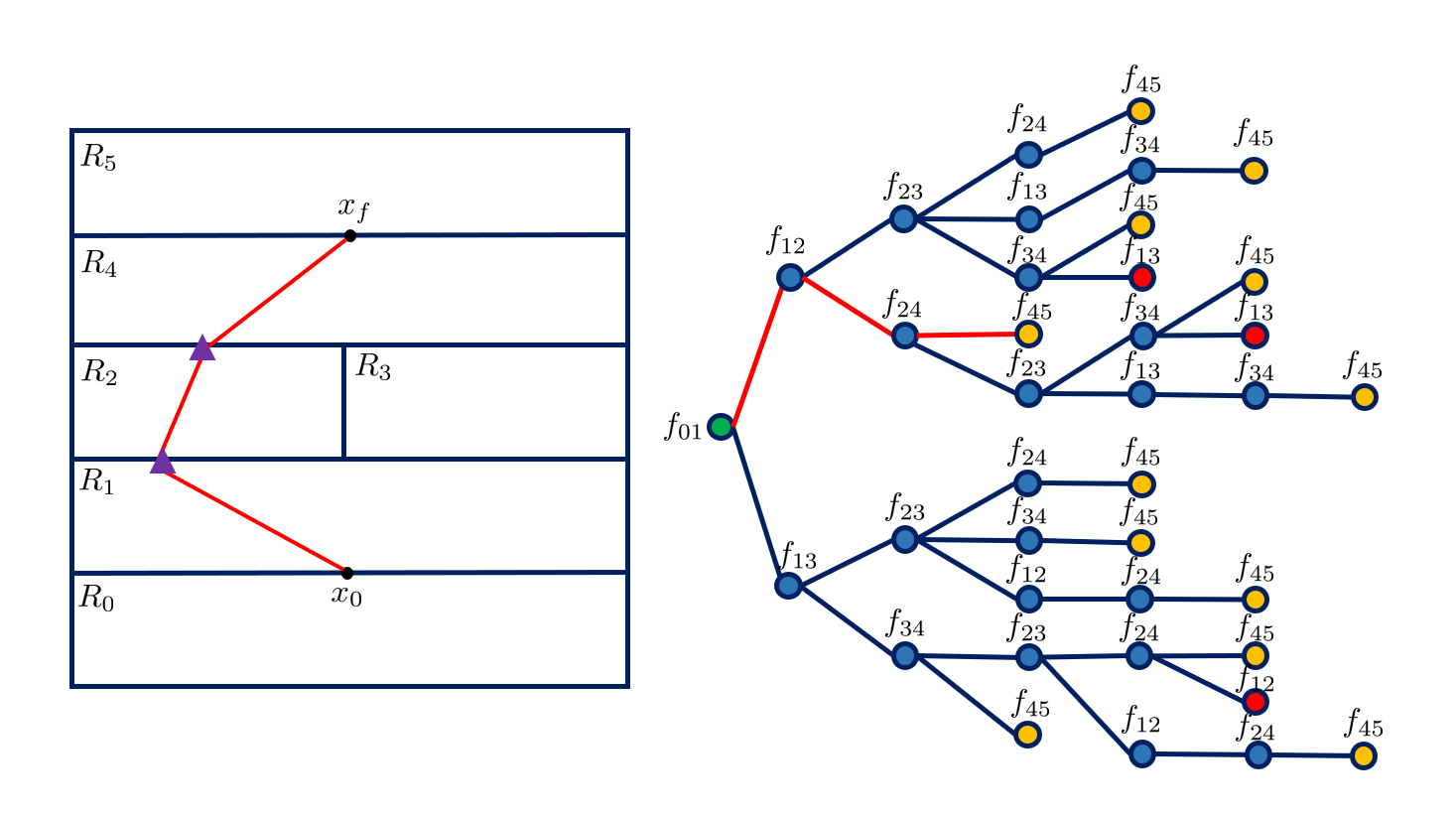}
    \caption{\textbf{(Left):} Example of a partitioned space. The domain is partitioned into 6 cells. the red line represents one feasible path from the start to the destination position, with the junctions represented by the purple triangles. \textbf{(Right):} The entire decision tree for this partitioned space. The root node shown by the green circle is the boundary curve containing the starting position, and the terminal node (yellow circle) represents the boundary curve containing the destination node. For each node, it is connected to its child node if it is an adjacent boundary to its child.  The red nodes represent the nodes with no unvisited neighboring nodes. The red path in the graph corresponds to the cell sequence crossed by the feasible path shown on the left figure.}
    \label{fig:tree_struct}
\end{figure*}

\begin{algorithm}[h]
\setstretch{0.8}
\label{alg_1}
\caption{Main algorithm}
\KwData{initial position $x_0$, final position $x_f$, partitioned cell $\{R_\alpha\}_{\alpha\in I_R}$}
\KwOut{optimal junction position $x(\lambda^{opt})$}
visited $\leftarrow\{\text{FALSE}\}$\;
$\mathcal{C}\leftarrow \{\}$\;
CostArrival $=0$\;
TotalCost\_ub $= 0$\;
Start node $s = [c_0, c_1]$, goal node $d = [c_n, c_{n+1}]$\;
$\lambda^{opt}$ = findAllCellSeq($s$, $d$, visited, $\mathcal{C}$, adjacency, $f_g(s)$, TotalCost\_ub)\;

\texttt{\\}
\SetKwFunction{Ffindallpath}{\textbf{findAllCellSeq}}
\SetKwProg{Fn}{Function}{:}{}
\Fn{\Ffindallpath{$n_c$, d, visited, $\mathcal{C}$, $f_g(n_c)$, TotalCost\_ub}}{
visited($n_c$) = TRUE\;
$\mathcal{C}$.append($n_c$)\;
\If{$n_c=d$}{
$\lambda =$ Intermittent\_Diffusion($\mathcal{C}$)\label{line12}\;
\If{$J(\lambda)< $ TotalCost\_ub or TotalCost\_ub $=0$}{
TotalCost\_ub $\leftarrow J(\lambda)$\label{line14}\;
$\lambda^{opt} = \lambda$\;}
}
\Else{
\For{all adjacent node $\{m_j\}_{j=1}^M$ of $n_c$}{
\If{visited$(m_j)$ = FALSE}{
Compute $g_{j, lb}$ using \eqref{travel_time_minimum} or \eqref{energy_minimum}\;
Compute $f_g(m_j)$ using \eqref{lb_coa}\;
\If{$f_g(m_j) >$ TotalCost\_ub and TotalCost\_ub $\neq 0$}{
continue\label{line24}\;} 
findAllCellSeq($m_j$, d, visited, $\mathcal{C}$, $f_g(m_j)+g_{j, lb}$, TotalCost\_ub)\label{line26}\; 
}
}
$\mathcal{C}$.pop\;
visited(u) = FALSE\;
}
}
\end{algorithm}

\begin{algorithm}[h]
\label{alg_2}
\setstretch{0.8}
\caption{Intermittent Diffusion}
\KwData{cell sequence $\mathcal{C}$}
\KwOut{the optimal junction position $\lambda^*$ given the fixed cell sequence $\mathcal{C}$}
Initialize $\lambda^0 = 0$\;
Set evolution step number $N$\;
Choose threshold $\epsilon$\;
\For{$i=1,\cdots,N$}{
  Choose perturbation duration $T_i$\;
  Choose perturbation intensity $\sigma_i$\;
  \For{$j=1,\cdots,T_i$}{
    Update $\lambda^i$ using \eqref{update}\;
  }
  Set $\sigma_i=0$\;
  \While{not converges}{
    Update $\lambda^i$ using \eqref{update}\;
  }
}
Set $\lambda^* =\arg\min_{i\leq N}J(\lambda^i)$\;
\end{algorithm}

\section{Completeness}\label{completeness}
In this section, we demonstrate that Algorithm \ref{alg_1} is complete if $L(v)$ is a convex function of $v$. 
\begin{theorem}\label{converge_thm}
If the flow field is piece-wise constant and
\begin{equation}\label{flow_v_bound}
\max_{\alpha\in I_R}\|u_\alpha\|<V,
\end{equation}
let $Q$ be the set of global minimizers, $U$ be a neighborhood of $Q$. Then for any $\epsilon>0$, there exists $T_0,N_0,\sigma_0$ such that if $T_i>T_0$, $\sigma_i<\sigma_0$ (for $i=1,2,\cdots,N$) and $N>N_0$ where $T_i,\sigma_i,N$ are parameters in Algorithm \ref{alg_2}, $\mathbb{P}(\lambda^{opt}\in U)\geq1-\epsilon$, where $\lambda^{opt}$ is the optimal solution found by Algorithm \ref{alg_1}. Thus, Algorithm \ref{alg_1} is complete.
\end{theorem}

The idea is that by Bellman principle, optimal trajectory admits an optimal sub-structure property, that is, any piece of the optimal trajectory is also optimal for the sub-problem. By applying this principle, we consider the path segment in each single region, and try to construct a solution $\psi$ with two types of objective function described in Section \ref{problem_statement}, for the Hamilton-Jacobi-Bellman equation (HJB)
\begin{equation} \label{HJB}
\psi_t(x,t)+H(x,\nabla\psi(x,t))=0
\end{equation}
where
\begin{equation*}
H(x,p)=\max_{\|v\|\leq V}\left\lbrace p^T (v+u)-L(x,v)\right\rbrace.
\end{equation*}
is the Hamiltonian and
\begin{equation*}
\begin{aligned}
\psi(x,t)=\min_{v}\Big\lbrace&\int_0^t L(\gamma,v)ds:\dot{\gamma}=v+u,\gamma(0)=x_0,\\
&\gamma(t)=x,\max_{s\in[0,t]}v(x)\leq V\Big\rbrace
\end{aligned}
\end{equation*}
is the value function. Since the original problem takes the minimum over all possible time, we take $\min_t\psi(x,t)$ to get the optimizer in the given region and claim that the corresponding motion gives a global optimal for the sub-problem in the single region. Hence in the following we give optimality proof of solution of the sub-problem in each constant flow region. Given the optimal solution of the sub-problem in each constant flow region, proof of Theorem \ref{thm:general} is provided in the Appendix.

\subsection{Total Travel Time}\label{completeness_time}
$L(x,v)=1$ for total travel time minimization. To construct the value function at the point $(x,t)$, we introduce the maximum speed constant velocity motion in the region with flow velocity $u$, that is, in this region, the vehicle moves in straight line from $x_0$ to $x$ with velocity $v+u$ and $\|v\|=V$, $\|v+u\|$ is given by \eqref{v_plus_u_norm}. We claim that
\begin{lemma}\label{time_lemma}
In a constant flow field, the maximum speed straight line motion is optimal if we minimize the total travel time
\begin{align*}
\min_{v,T}&\int_0^Tdt\\
s.t.\ \ \ &\dot{x}=v+u,\\
&x(0)=x_0,\\
&x(T)=x_f,\\
&\max_{t\in[0,T]}\|v\|\leq V.
\end{align*}
\end{lemma}

\subsection{Quadratic Energy with a constant running cost}\label{completeness_energy}

\begin{lemma}\label{energy_lemma}
In a constant flow field, the minimizer of energy optimal problem
\begin{align*}
\min_{v,T}&\int_0^T\|v\|^2+Cdt\\
s.t.\ \ \ &\dot{x}=v+u,\\
&x(0)=x_0,\\
&x(T)=x_f,\\
&\max_{t\in[0,T]}\|v\|\leq V
\end{align*}
is the constant velocity motion in the speed of $\|v\|$, where
\begin{equation*}
\|v\|=\left\lbrace
\begin{array}{ll}
\begin{aligned}
\Big(&C+2\|u\|^2\\&-\sqrt{C+\|u\|^2}\frac{2(x-x_0)^Tu}{\|x-x_0\|}\Big)^{1/2}
\end{aligned}
&\text{if \eqref{equ:energy_equation} holds}\\
V&\text{otherwise}
\end{array}
\right..
\end{equation*}
\end{lemma}
\begin{remark}\label{remark_41}
When \eqref{equ:energy_equation} does not hold, the minimizer of the energy optimal problem is the same as the minimizer of travel time optimal one. Hence, if the constant running cost $C$ is large enough, solving the energy optimal problem is equivalent to solving the travel time optimal problem.
\end{remark}

Based on the proof of Lemma \ref{time_lemma} and Lemma \ref{energy_lemma}, we can have the following theorem, which tells the optimal path structure within each constant flow region, given entrance and exit locations.
\begin{theorem}\label{inter_thm}
In each constant flow region, given the entrance and exit locations, the vehicle motion defined in Lemma \ref{time_lemma} and Lemma \ref{energy_lemma} solves the HJB equation
\begin{equation*}
\psi_t(x,t)+\max_v\{\nabla\psi(x,t)^T(u+v)-L(x,v)\}=0
\end{equation*}
for $L=1$ and $L=\|v\|^2+C$ respectively. Moreover, among all the solutions of the above HJB, motion in these two lemmas gives the path with shortest time/minimum energy. Thus, we have the optimal solution of the sub-problem in each region.
\end{theorem}


\subsection{A general convex Lagrangian}
In general, if we only assume the Lagrangian $L=L(v)$ is a convex function and the dynamics is $\dot{x}=f(u+v)$ where $f$ is invertible and $0$ is in the range of $f$ (there exists some $y$ with $\|y\|<+\infty$ such that $f(y)=0$), we can have similar optimal path within a constant flow field and the result is stated in Theorem \ref{thm:general}.

\section{Simulation Results}\label{experiments}
In this section, we provide simulations to validate the strength of the proposed method. First, the time-optimal and energy-optimal path planning examples with vehicle travel in simple canonical time flow field are presented. This example serves as a benchmark example wherein, we compare the solution obtained by our algorithm to solution derived from other path planning methods. Then we present a path planning example of using the proposed method to plan the time-optimal and energy-optimal path in a realistic ocean surface flow field. This simulation is intended to verify the performance of the proposed method in a highly complicated and strong real ocean flow field.

\subsection{Jet flow in 3D space}\label{3djet}
For this benchmark example, we present path planning using the proposed method in a jet flow in $3D$ space. The domain consists of three regions, divided by two boundary surfaces, $z = 10$ and $z = 15$. In the region where $z\in(0, 10)$, the flow speed is $(0.5, 0, 0)$. There is strong jet flow in the region where $z\in(10, 15)$ with flow speed $(2, 1, 0)$. The flow speed is zero in the region where $z\in(15, 20)$. The starting position is assigned at the origin, while the goal position is assigned at $(0, 0, 20)$.
\begin{figure*}[!t]
  \centering
 \includegraphics[width=0.4\textwidth]{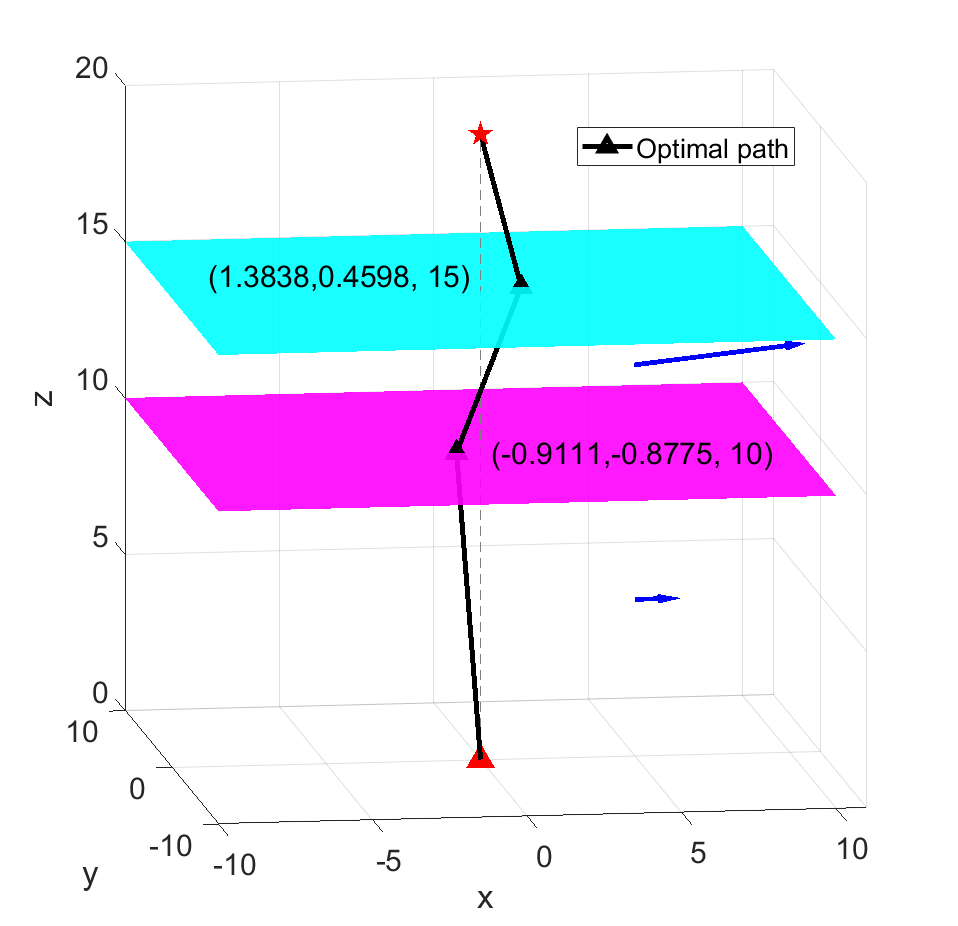}
 \includegraphics[width=0.4\textwidth]{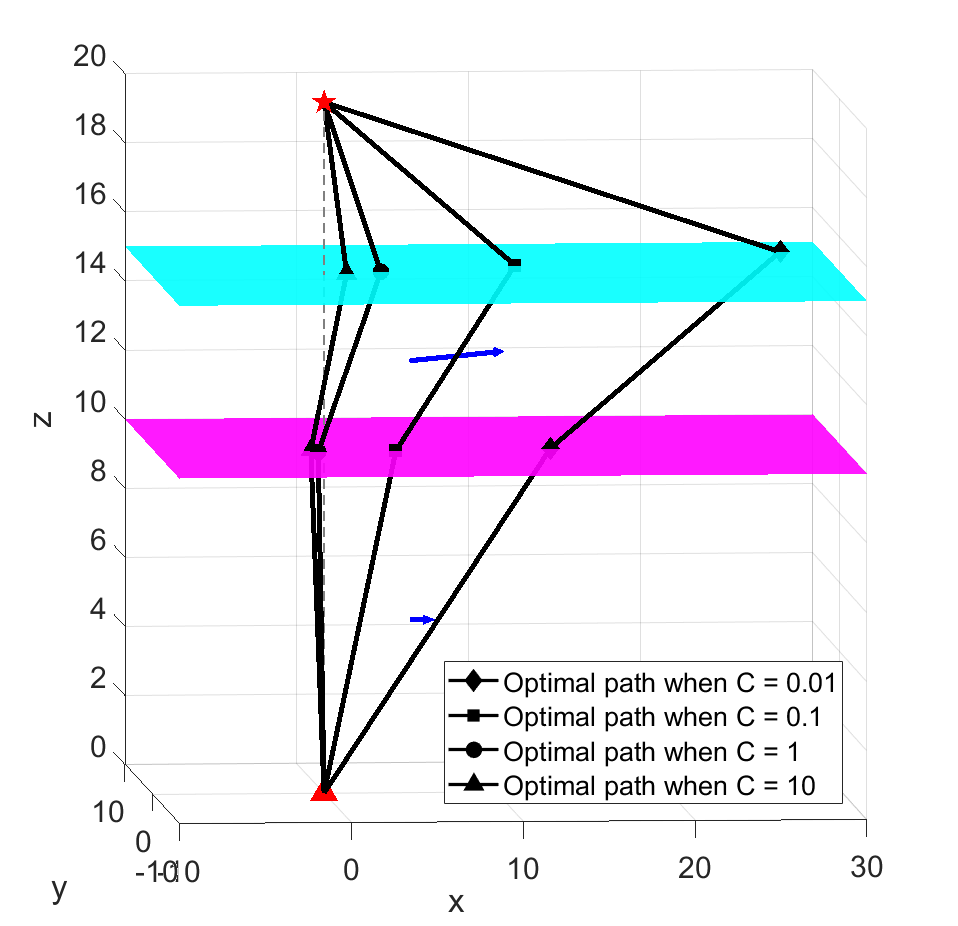}
  \caption{\textbf{(Left)}: Time optimal path planned by the proposed method. \textbf{(Right)}: Energy optimal path planned by the proposed method.. In both plots, boundaries of the jet flow are denoted by the colored surfaces. The flow speed in the domain is represented by the blue arrows. The optimal path is marked by black line, while the marker position denotes junction points computed by the proposed method.}
  \label{sim3d}
\end{figure*}

The left figure in Fig.~\ref{sim3d} shows the time-optimal path planned by the proposed method. The time-optimal solution is compared with the time-optimal path planned by the LSM. 
The comparison result is shown in Table \ref{sim5cmp_to}. In this comparison, assuming the path segment $x_{i+1} - x_i$ travels from the boundary surface $f_{\alpha_i\beta_i}$ to reach the boundary surface $f_{\alpha_{i+1}\beta_{i+1}}$, we define $\theta_i$ as the angle between path segment $x_{i+1} - x_i$ and the boundary surface $f_{\alpha_i\beta_i}$. $\gamma_i$ is defined as the angle between the projection of $x_{i+1} - x_i$ on the boundary surface $f_{\alpha_i\beta_i}$ and the x-axis of $f_{\alpha_i\beta_i}$, $\theta_i \in (0, 90\degree], \gamma_i \in (-180\degree, 180\degree]$. From the table, $\theta_i$ and $\gamma_i$ computed from the proposed method and the LSM are similar, with approximately $1\degree$ difference. Travel time of the optimal path planned by the proposed method and LSM are also approximately the same. This shows that the proposed algorithm converges to the optimal solution.


\begin{table}[]
\begin{center}
\caption{Comparison between using the proposed method, and LSM for time-optimal path planning}
\begin{tabular}{||c|c|c||}
\hline
  &   Proposed Method & LSM \\
  \hline\hline
 $\theta_1$  & $82.7924$ & $83.5659$ \\
 \hline
 $\theta_2$     & $62.0255$ & $63.3118$  \\
 \hline
 $\theta_3$    & $73.7397$ & $73.8027$ \\
 \hline
 $\gamma_1$  & $-136.0775$ & $-135.6592$ \\
 \hline
 $\gamma_2$  & $30.2293$ & $30.2407$ \\
 \hline
 $\gamma_3$  & $-161.6199$ & $-161.2246$ \\
 \hline
Total cost & $6.9096$ & $6.9826$ 
\\
 \hline
\end{tabular}
\label{sim5cmp_to}
\end{center}
\end{table}

The energy-optimal path is shown in the right figure of Fig.~\ref{sim3d}. The energy-optimal path when $C = 10$ is exactly the same as the time-optimal path. As the assigned $C$ decreases, the energy cost is attached relatively more weight in the cost function. Therefore, the vehicle tends to save more energy to go with the flow in the bottom region and the jet flow region. Thus, the energy-optimal path deviates from the time-optimal path as $C$ decreases.

\subsection{Surface ocean flow}\label{realflow}

In this section we present path planning simulation of an underwater glider traveling in real ocean surface flow field near Cape Hatteras, North Carolina, a highly dynamic region characterized by confluent western boundary currents and convergence in the adjacent shelf and slope waters. While deployed, the glider is subject to rich and complex current fields driven by a combination and interaction of Gulf Stream, wind, and buoyancy forcing, with significant cross-shelf exchange on small spatial scales that is highly challenging for planning algorithms. 
While the energy efficiency of the glider's propulsion mechanism permits endurance of weeks to months, the forward speed of the vehicles is fairly limited (0.25-0.30 m/s), which can create significant challenges for navigation in strong currents.  Use of a thruster in a so-called ``hybrid" glider configuration can increase forward speed to approximately 1 m/s \cite{ji2019design}, but at great energetic cost.  The continental shelf near Cape Hatteras is strongly influenced by the presence of the Gulf Stream, which periodically intrudes onto the shelf, resulting in strong and spatially variable flow that can be nearly an order of magnitude greater than the forward speed of the vehicle (2+ m/s).  Due to the high flow speed, we consider the deployment of a hybrid underwater glider in this simulation, and consider the vehicle speed $V = 1$ m/s.

The input flow map for path planning is given by a 1-km horizontal resolution version of the Navy Coastal Ocean Model (NCOM) \cite{martin2000} made available by J. Book and J. Osborne (Naval Research Laboratory, Stennis Space Center). The domain contains $130\times130$ grid points. One snapshot of the dynamic flow field is shown in  Fig.~\ref{sim_plot}. 
We partition the flow field using the algorithm proposed in \cite{hou2019partitioning}. The boundaries of the divided regions are shown in Fig.~\ref{sim_plot}.

We perform 3 sets of simulation, with each set contains 10 test cases. The start and destination position are chosen such that the distance between the two points is 40, 100, or 130 km. Fig. \ref{sim_plot} shows one test case, where $d = 100$ km. 
To verify performance of the proposed algorithm, we compare its simulation result with the A* and the LSM. Both A* and LSM run on the $130\times 130$ rectangular grid cells. To avoid the incompleteness issue of A* \cite{kularatne2018going, Soulignac2011}, in the node generation process of A*, we consider each grid point have 16 neighboring nodes. Comparison between the 3 algorithms is shown in Table \ref{tab:sim}. We compare the computation cost of the three algorithms, and the total travel cost of the optimal path derived by the three algorithms. Note that for the proposed algorithm, even though the optimal path is computed in the partitioned flow field, we compute the total travel cost of the vehicle tracking the optimal path in the original flow field given by NCOM. For the 3 set of simulations, the computation cost and total travel cost is averaged over the 10 test cases. 

For all 3 sets of simulation, the proposed algorithm takes less computation time to compute the optimal solution to the minimum-time planning problem. The proposed algorithm takes significantly less computation time in the simulation sets with shorter distance between the start and the destination node. The reason is that with smaller $d$, the decision tree is shallow, and the proposed algorithm only need to search through a small number of nodes to find the optimal solution. The total cost of optimal path computed by the proposed algorithm is comparable to the optimal solution from the A* and the LSM method.

\begin{figure*}[ht]
  \centering
\includegraphics[width=0.4\textwidth]{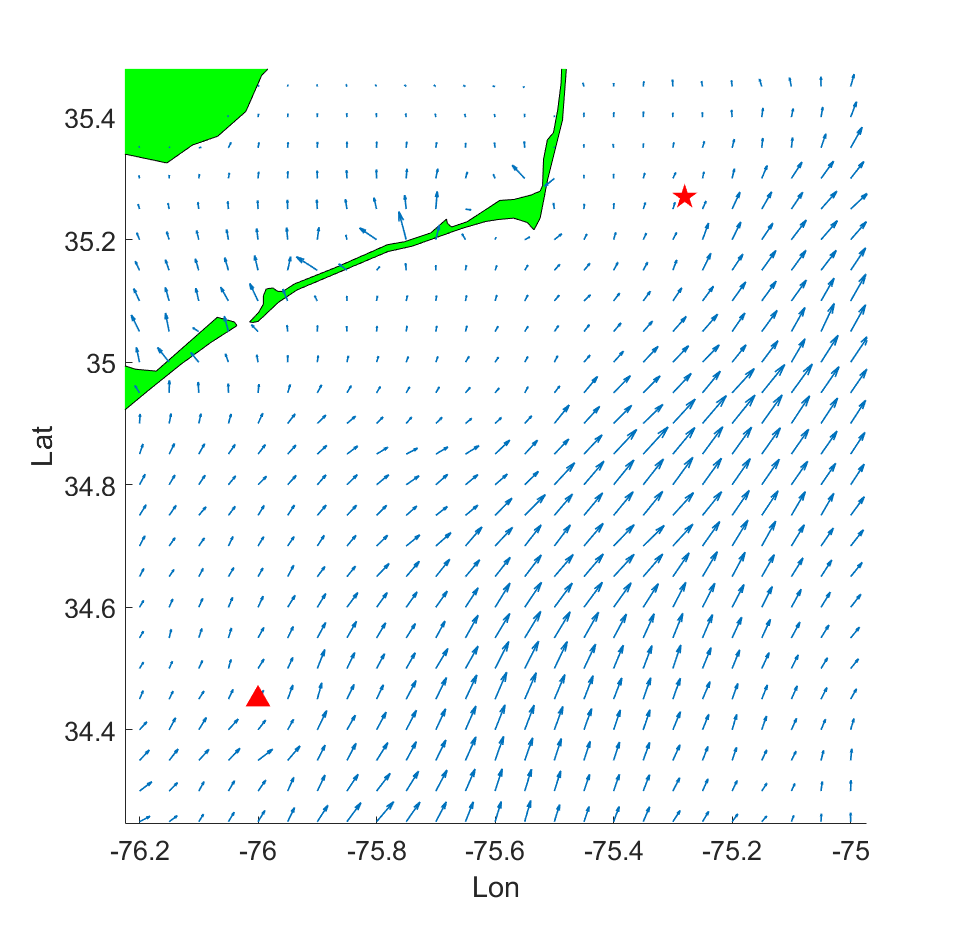}
\includegraphics[width=0.4\textwidth]{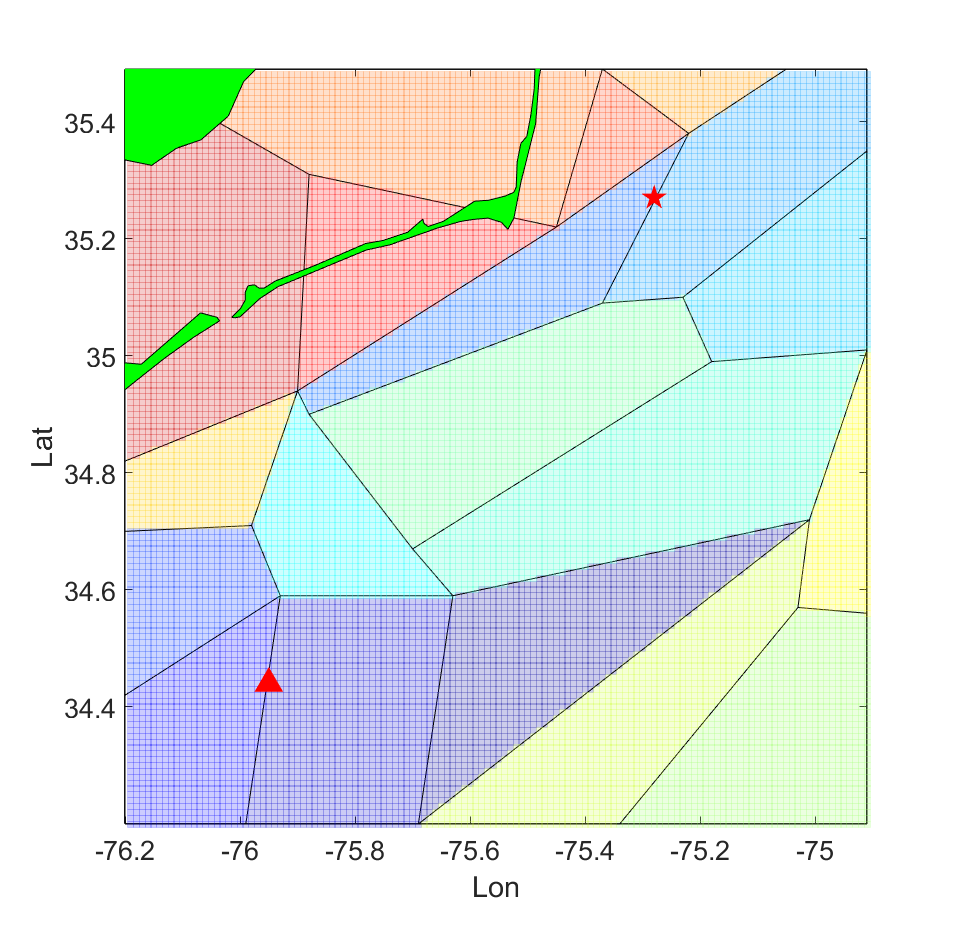}
  \caption{\textbf{(Left)}: Surface ocean flow field on May 27, 2017, 00:00 UTC at Cape Hatteras, NC. Red triangle and star indicate the starting and goal  position (when $d = 100$ km). \textbf{(Right)}: Partitioned deployment region. The partitioned cells are represented by the colored cells.. The original and the partitioned flow field. }
  \label{sim_plot}
\end{figure*}

\begin{figure*}[ht]
  \centering
\includegraphics[width=0.4\textwidth]{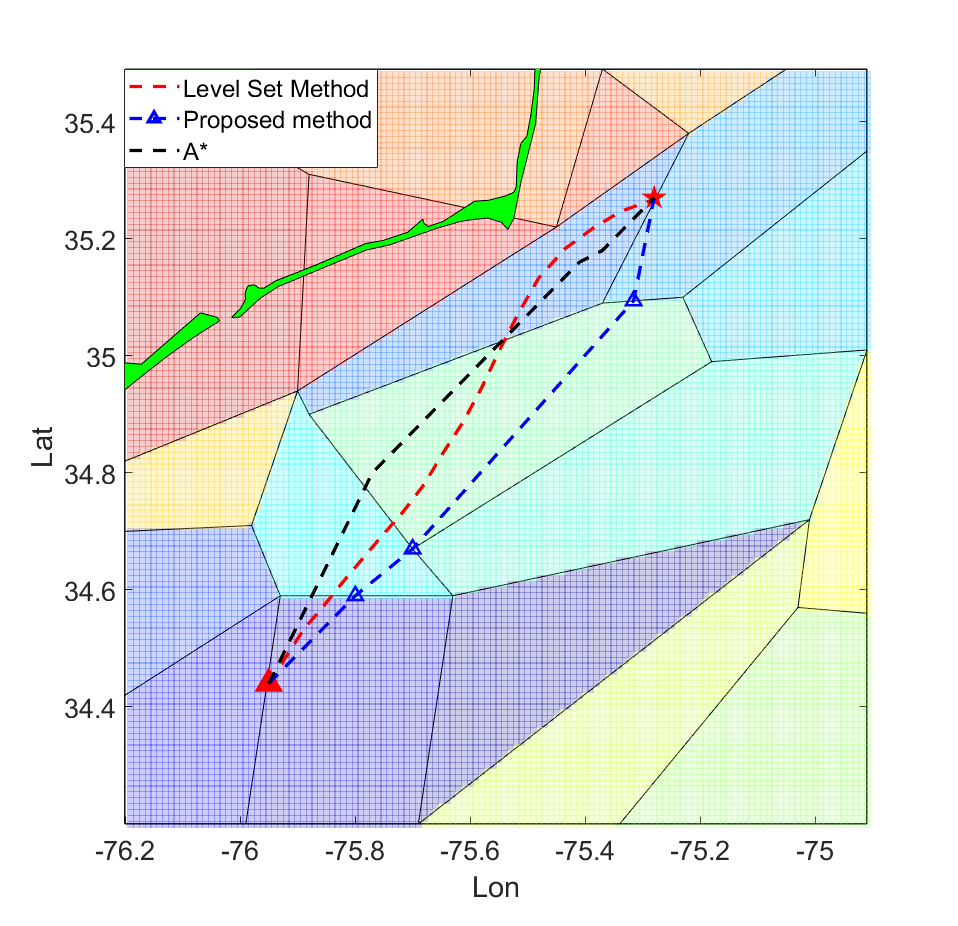} 
\includegraphics[width=0.4\textwidth]{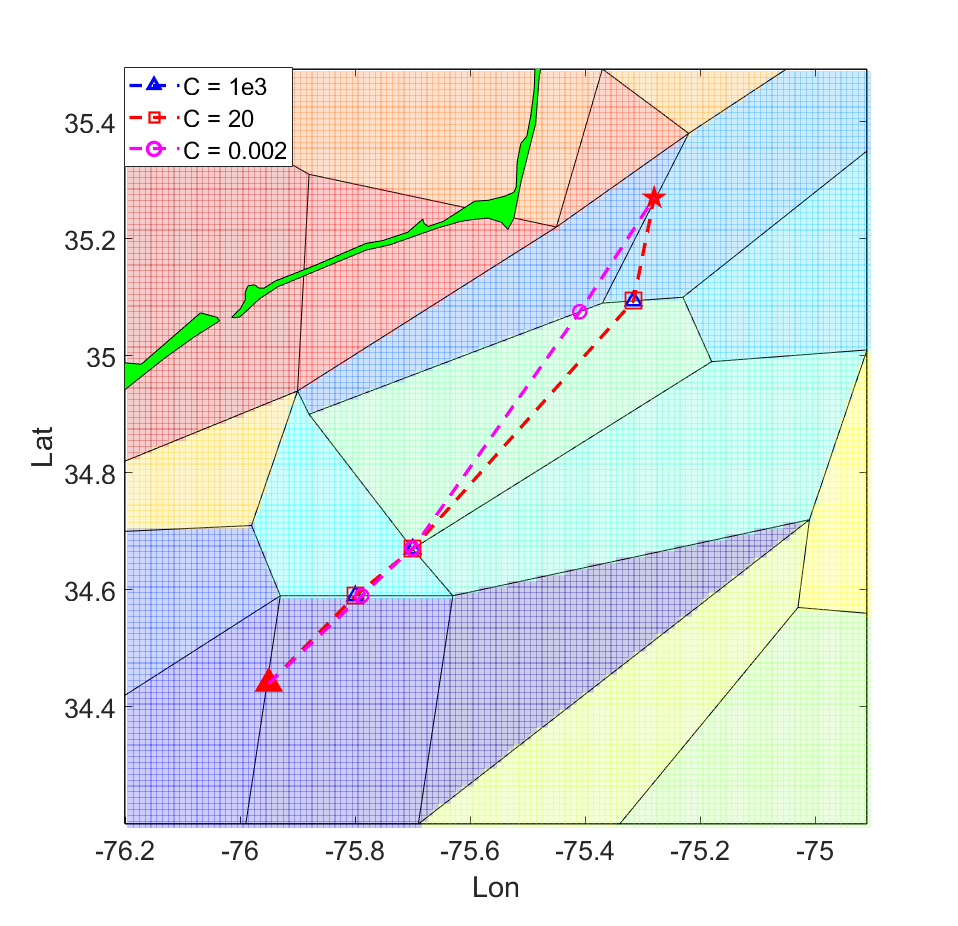} 
  \caption{\textbf{(Left)}: Time optimal path planned by the proposed method, the A* method and the LSM. \textbf{(Right)}: Energy optimal path planned by the proposed method, when $C$ takes different value. The optimal path is marked by colored line, while the marker position denotes junction points computed by the proposed method.}
  \label{sim4}
\end{figure*}

\begin{table*}[h]
\centering
\caption{Comparison between using the proposed method and existing algorithms for time-optimal path planning}
\begin{tabular}{ c c c c c }
\toprule
                        &                 & LSM           & A*       & Proposed alg. \\ 
                        \botrule
\multirow{2}{*}{40 km}  & Comp Time        &  $85.625$ s &  $25.7969$ s  &     $15.3225$ s     \\
\cline{2-5}
                        & Total Cost       &  $24.6594$ hrs   &  $26.4190$ hrs   & $27.3210$ hrs         \\ \cline{1-5}
\multirow{2}{*}{100 km} & Comp Time        & $281.938$ s   & $169.313$ s & $84.921$ s    \\ 
\cline{2-5} 
                        & Total Cost       & $35.4476$ hrs & $35.8012$ hrs & $36.1229$ hrs \\ \cline{1-5}
\multirow{2}{*}{130 km} & Comp Time &   $195.625$ s    &    $173.813$ s             &  $93.4643$ s   \\ 
\cline{2-5} 
                        & Total Cost       &   $ 38.5701$ hrs    &  $36.9645$ hrs  &  $40.3937$ hrs         \\ 
\botrule
\end{tabular}\label{tab:sim}
\end{table*}

The time-optimal and energy-optimal paths are shown in Fig. \ref{sim4}. For the time-optimal path, the planned path computed by the proposed algorithm takes a slight detour towards the off-shore direction to take advantage of the high ocean flow speed towards the North. Since the input to the proposed algorithm is the partitioned flow field, while the input to the LSM algorithm is the grid represented flow field, the optimal path computed by the proposed algorithm is different from the result of LSM. However, as shown in Table \ref{tab:sim}, solution quality of both the proposed algorithm and the LSM are comparable. In Table \ref{tab:sim} we present the total travel cost of the planned path for both LSM and the proposed algorithm. The total travel cost is computed in the actual flow field, instead of the partitioned flow field. As shown in the table, the total cost of the planned path computed by the two algorithms are comparable, while the proposed algorithm spends less computation time than the LSM. 

Energy optimal path planning generates same results when $C = 1e03$ and $C = 20$. In these cases, the running cost is much larger than the vehicle speed. Thus the energy-optimal planned path is identical to the time-optimal path. When $C = 2e-3$, the running cost is much less than vehicle speed. In this case, instead of making use of the strong jet flow, the proposed method generates planned path that go straight towards the goal position.

\section{Conclusion} 
In this paper, we propose a new method using Method of Evolving Junctions to solve the AUV path planning problem in an arbitrary flow field with the dynamics being $\dot{x}=u+v$ where $u$ is the flow field and $v$ is the vehicle velocity. Taking advantage of the explicit solution in constant flow field being straight line motion, we partition the flow field into piece-wise constant vector field and transform the optimal control problem into a finite dimensional optimization, using Intermittent Diffusion method to get the global minimizer. In this way, we can get rid of the system error induced by discretizing the continuous space. Also, our method can be trivially extended to high dimensional general vehicle path planning problems in the same time complexity without making further assumption.

\section*{Acknowledgment}

The authors would like to thank the support from NSF grants DMS-1830225,  ONR grant N00014-21-1-2891, 
ONR grants  N00014-19-1-2556 and N00014-19-1-2266;  AFOSR grant FA9550-19-1-0283; NSF
grants  CNS-1828678, S\&AS-1849228 and GCR-1934836; NRL grants N00173-17-1-G001 and N00173-19-P-1412 ; and NOAA grant NA16NOS0120028.

\section*{Appendix}\label{appendix}
In this appendix, we give proofs for the properties related to Algorithm \ref{alg_1}.

First we present proof of Theorem \ref{converge_thm}. The proof leverages the following theorem in \cite{chow2013global}.
\begin{theorem}\label{thm:ID}
Let $Q$ be the set of global minimizers, $U$ be a small neighborhood of $Q$ and $\lambda_{opt}$ the optimal solution obtained by the ID process. Then for any given $\epsilon>0$, there exists $\tau>0$, $\sigma_0>0$ and $N_0>0$ such that if $T_i-S_i>\tau$, $\sigma_i<\sigma_0$ (for $i=1,\cdots,N$) and $N>N_0$,
\begin{equation*}
\mathbb{P}(\lambda_{opt}\in U)\geq1-\epsilon.
\end{equation*}
\end{theorem}
Then we provide completeness proof of Algorithm \ref{alg_1}. 

\begin{proof}[Proof of Theorem \ref{converge_thm}:]
The proof includes two steps. First, we show that the decision tree returns all cell sequences with total cost less than or equal to the lowest-cost path found so far. Then we prove that given a fixed cell sequence, the global minimizer can be found by Algorithm \ref{alg_2}.

The DFS algorithm, which avoids repeated states in the graph, is complete in finite state spaces \cite{russell2002artificial}. In a static flow field divided into convex regions, the optimal path will not visit a cell boundary curve more than one time. Hence, the optimal path connecting the root and the target node in the decision tree does not contain loops. Therefore, the BnBDFS returns all cell sequences with total cost less than or equal to the lowest-cost found so far.

Next we show that the ID algorithm is complete.
We combine Theorem \ref{inter_thm}, together with Bellman principle, to show that the global optimal path must be in the structure of constant motion within each flow region. To prove that the proposed algorithm is convergent, we only need to show that there exists a global minimizer $\lambda^*=(\lambda_1^*,\cdots,\lambda_K^*)$, around which there is a closed neighborhood $U\subset\prod_{i=1}^KD_i$ such that $vol(U)>0$ ($vol$ is the product Lebesgue measure in $\prod_{i=1}^KD_i$) and for all $\lambda\in U$, the gradient flow $\dot{\lambda}=-\nabla J(\lambda)$ converges to $\lambda^*$. If this condition holds, we can follow the proof of intermittent diffusion and get the desired results.

To this end, if there exists such $U$ that $vol(U)>0$ and for all $\lambda\in U$, we have $J(\lambda)\leq J(\mu)$ for arbitrary $\mu\in S$ for some $S\subset U$, then the proof is done. Now if the global minimizers are isolated, then given any global minimizer $\lambda^*=(\lambda_1^*,\cdots,\lambda_K^*)$, since $J$ is continuous differentiable, we can have a closed neighborhood $U\subset\prod_{i=1}^KD_i$ with $vol(U)>0$ (within the neighborhood, the dimension of the domain does not change) such that $J(\lambda)>J(\lambda^*)$ and $\nabla J(\lambda)\neq0$ for all $\lambda\in U\backslash\{\lambda^*\}$, then the gradient flow starting at $\lambda\in U$ converges to $\lambda^*$. 

Therefore, we prove the algorithm is complete.
\end{proof} 


\begin{proof}[Proof of Lemma \ref{time_lemma}]

We write the value function as
\begin{equation*}
\psi(x,t)=\frac{\|x-x_0\|}{\|v+u\|}.
\end{equation*}
To make the problem complete, we define $\psi(x,t)=+\infty$ if the vehicle cannot reach $x$ in time $t$, which gives the final value function to be
\begin{equation*}
\psi(x,t)=\left\lbrace\begin{array}{ll}\frac{\|x-x_0\|}{\|v+u\|}&\frac{\|x-x_0\|}{\|v+u\|}\leq t\\+\infty&\text{otherwise}\end{array}\right..
\end{equation*}
If only the reachable part is considered, from the above equation, we can calculate $\psi_t=0$ and
\begin{equation}\label{gradpsi}
\begin{aligned}
\nabla\psi=\frac{1}{\|v+u\|^2}\Big(&\|v+u\|\frac{x-x_0}{\|x-x_0\|}\\&-\|x-x_0\|\nabla\|v+u\|\Big).
\end{aligned}
\end{equation}
 We can rewrite $v=v^0+v^\perp$ and $V^2=\|v^0\|^2+\|v^\perp\|^2$ if we denote
\begin{eqnarray*}
v^0&=&\frac{(x-x_0)}{\|x-x_0\|}\frac{(x-x_0)^Tv}{\|x-x_0\|},\\
v^\perp&=&\left(I-\frac{(x-x_0)}{\|x-x_0\|}\frac{(x-x_0)^T}{\|x-x_0\|}\right)v,
\end{eqnarray*}
where $I$ is the identity matrix. And $u$ can be decomposed in the same manner $u=u_0+u^\perp$. It is easy to see that $v^\perp=-u^\perp$ since $(v+u)/\|v+u\|=(x-x_0)/\|x-x_0\|$. Then, we see that $\|v+u\|=\|v_0\|+\|u_0\|$ and
\begin{equation*}
\begin{aligned}
&\sqrt{\left(\frac{(x-x_0)^Tu}{\|x-x_0\|}\right)^2+V^2-\|u\|^2}\\
=&\Big(\|u^0\|^2+\|v^0\|^2+\|v^\perp\|^2-\|u^0\|^2-\|u^\perp\|^2\Big)^{1/2}\\
=&\|v^0\|.
\end{aligned}
\end{equation*}
Hence, we have
\begin{align*}
&\nabla\|v+u\|\\
=&\nabla \left(\frac{(x-x_0)^Tu}{\|x-x_0\|}+\sqrt{\left(\frac{(x-x_0)^Tu}{\|x-x_0\|}\right)^2+V^2-\|u\|^2}\right)\\
=&\frac{\|v+u\|}{\|v^0\|}\nabla\frac{(x-x_0)^Tu}{\|x-x_0\|}\\
=&\frac{\|v+u\|}{\|x-x_0\|\|v^0\|}\left(I-\frac{(x-x_0)}{\|x-x_0\|}\frac{(x-x_0)^T}{\|x-x_0\|}\right)u\\
=&\frac{\|v+u\|}{\|x-x_0\|\|v^0\|}u^\perp.
\end{align*}
Taking $\nabla\|v+u\|$ back to \eqref{gradpsi} and noticing that $v^\perp=-u^\perp$, we reduce the gradient to be
\begin{eqnarray*}
\nabla\psi&=&\left(\frac{x-x_0}{\|x-x_0\|}-\frac{u^\perp}{\|v^0\|}\right)\frac{1}{\|v+u\|}=\frac{v}{\|v^0\|\|v+u\|}.
\end{eqnarray*}
With the above equation, the Hamiltonian is
\begin{align*}
H&=\sup_{\hat{v}:\|\hat{v}\|\leq V}\left(\nabla\psi^T(\hat{v}+u)-1\right)\\
&=\sup_{\hat{v}:\|\hat{v}\|\leq V}\left\lbrace\frac{v^T}{\|v^0\|}\frac{\hat{v}+u}{\|v+u\|}\right\rbrace-1\\
&=\frac{1}{\|v+u\|}\left(\sup_{\hat{v}:\|\hat{v}\|\leq V}\left\lbrace\frac{v^T\hat{v}}{\|v^0\|}\right\rbrace+\|u_0\|-\frac{\|u^\perp\|^2}{\|v^0\|}\right)-1\\
&=\frac{1}{\|v+u\|}\left(\frac{V^2}{\|v^0\|}+\|u^0\|-\frac{\|u^\perp\|^2}{\|v^0\|}\right)-1\\
&=\frac{1}{\|v+u\|}(\|v^0\|+\|u^0\|)-1=0,
\end{align*}
which leads to the conclusion that the value function induced by the maximum speed constant velocity motion solves the Hamilton-Jacobi equation, thus is the optimal moving pattern in a constant flow speed region since $\min_t\psi=\psi$.
\end{proof}

Meanwhile, using the same notation and logic, we can give the proof of Proposition \ref{differentiable_prop_energy}:

\begin{proof}[Proof of Proposition \ref{differentiable_prop_energy}]
First we show that the objective function is well-defined if there exists a feasible trajectory, and $\|u_{c_i}\|\leq V$.
If $(x_i-x_{i-1})^T u_{c_i}\leq0$, unless $V>\|u_{c_i}\|$, there does not exists a feasible path. Therefore,
\begin{equation*}
\begin{aligned}
&(x_i-x_{i-1})^Tu_{c_i}<\\&\sqrt{\left((x_i-x_{i-1})^Tu_{c_i}\right)^2+\|x_i-x_{i-1}\|^2(V^2-\|u_{c_i}\|^2)}.
\end{aligned}
\end{equation*}
Since $\|u_{c_i}\|^2-V^2<0$, we have $t_i^*>0$. On the other hand, if $(x_i-x_{i-1})^Tu_{c_i}>0$, we can have two cases: $V>\|u_{c_i}\|$, which shares the same conclusion as the first case, and $V<\|u_{c_i}\|$. In the latter circumstance, since $\|u_{c_i}\|^2-V^2>0$ and
\begin{equation*}
\begin{aligned}
&(x_i-x_{i-1})^Tu_{c_i}>\\&\sqrt{\left((x_i-x_{i-1})^Tu_{c_i}\right)^2+\|x_i-x_{i-1}\|^2(V^2-\|u_{c_i}\|^2)},
\end{aligned}
\end{equation*}
it is still true that $t_i^*>0$.

Meanwhile, When $(x_i-x_{i-1})^Tu_{c_i}>0$, $t_i^*>0$ still holds if $V=\|u_{c_i}\|$ and actually
\begin{align*}
&t_i^*=\lim_{V^2-\|u_{c_i}\|^2\rightarrow0}\frac{1}{\|u_{c_i}\|^2-V^2}\Big((x_i-x_{i-1})^Tu_{c_i}\\
&-\sqrt{\left((x_i-x_{i-1})^Tu_{c_i}\right)^2+\|x_i-x_{i-1}\|^2(V^2-\|u_{c_i}\|^2)}\Big)\\
&=\frac{\|x_i-x_{i-1}\|^2}{2(x_i-x_{i-1})^Tu_{c_i}}>0.
\end{align*}
However, if $(x_i-x_{i-1})^Tu_{c_i}\leq0$, $V=\|u_{c_i}\|$ becomes a singular point since there is no feasible path. Thus, in this case, we cannot formally solve the problem.

Since $g^t_i(x_i,x_{i-1})=g^t_i(x_i-x_{i-1})$ and $g^e_i(x_i,x_{i-1})=g^e_i(x_i-x_{i-1})$, we only need to consider the differentibility of
\begin{equation*}
g(a)=\left\lbrace\begin{array}{ll}
    g_1(a) & \text{if \eqref{energy_speed_condition} holds} \\
    g_2(a) & \text{otherwise}
\end{array}\right.
\end{equation*}
where
\begin{align*}
g_1(a)&=2\sqrt{\|u_{c_i}\|^2+C}\|a\|-2a^Tu_{c_i}\\&=2\|a\|\left(\sqrt{\|u_{c_i}\|^2+C}-\|u^0_{c_i}\|\right),\\
g_2(a)&=\frac{V^2+C}{\|u_{c_i}\|^2-V^2}\Big(a^Tu_{c_i}\\
-& \sqrt{(a^Tu_{c_i})^2+\|a\|^2(V^2-\|u_{c_i}\|^2)}\Big)\\&=\frac{(V^2+C)\|a\|}{\|u_{c_i}^0\|+\|v^0\|}.
\end{align*}
First of all, when equality in \eqref{energy_speed_condition} holds, we have
\begin{equation}\label{energy_u_plus_c}
\begin{aligned}
&\sqrt{\|u_{c_i}\|^2+C}=\|u_{c_i}^0\|\pm\|v^0\|\\
\Longrightarrow&\sqrt{\|u_{c_i}\|^2+C}=\|u_{c_i}^0\|+\|v^0\|.
\end{aligned}
\end{equation}
We take the plus sign since $\sqrt{\|u_{c_i}\|^2+C}\geq\|u_{c_i}\|$. Meanwhile from \eqref{energy_speed_condition} and \eqref{energy_u_plus_c}, we can derive the following equation
\begin{equation*}
V^2+C=2\|v^0\|(\|u_{c_i}^0\|+\|v^0\|).
\end{equation*}
Therefore, $g(a)$ is continuous. Similar calculations shows that
\begin{align*}
\nabla g_1&=2\left((\sqrt{\|u_{c_i}\|^2+C})\frac{a}{\|a\|}-u_{c_i}\right)=v^0-u_{c_i}^\perp=2v,\\
\nabla g_2&=\frac{V^2+C}{\|v^0\|\|v+u_{c_i}\|}v=2v,
\end{align*}
which gives us the desired result.
\end{proof}

\begin{proof}[Proof of Lemma \ref{energy_lemma}]

In the case of minimum energy planning, $L(x,v)=\|v\|^2+C$ where $C\geq0$ is a constant running cost. To calculate the optimal solution for the vehicle running from $x_0$ to the target $x$ in a constant flow velocity field, we again study the constant speed straight line motion. However in this circumstance, the vehicle may no longer travel with maximum speed, hence we take the travel time in the region into consideration. Suppose that the the vehicle moves from $x_0$ to $x$ in time $t$, we set the vehicle velocity to be
\begin{equation*}
v=\frac{x-x_0}{t}-u,
\end{equation*}
assuming that
\begin{equation}\label{energy_assumption}
\|v\|^2=\frac{\|x-x_0\|^2}{t^2}+\|u\|^2-\frac{2(x-x_0)^Tu}{t}\leq V^2.
\end{equation}
Then the value function is
\begin{equation*}
\begin{aligned}
\psi(x,t)&=(\|v\|^2+C)t\\&=\frac{\|x-x_0\|^2}{t}-2(x-x_0)^Tu+(C+\|u\|^2)t.
\end{aligned}
\end{equation*}
Further we take
\begin{align*}
\psi(x,t)=\left\lbrace\begin{array}{ll}
    \begin{aligned}
    &\frac{\|x-x_0\|^2}{t}\\&-2(x-x_0)^Tu+(C+\|u\|^2)t 
    \end{aligned}
    & \|v\|\leq V \\
    +\infty & \text{otherwise}
\end{array}\right..
\end{align*}
Then by direct calculation with the finite part of $\psi$, we have
\begin{align}
&\psi_t=C+\|u\|^2-\frac{\|x-x_0\|^2}{t^2},\label{energy_psi_set1}\\
&\nabla\psi=\frac{2(x-x_0)^Tu}{t}-2u,\label{energy_psi_set2}\\
&\|\nabla\psi\|^2=\frac{4\|x-x_0\|^2}{t^2}+4\|u\|^2-\frac{8(x-x_0)^Tu}{t}\label{energy_psi_set3}
\end{align}
The Hamilton-Jacobi equation is in the form of
\begin{equation}\label{energy_hjb}
\psi_t+\sup_{v:\|v\|\leq V}\left\lbrace\nabla\psi^T(v+u)-\|v\|^2-C\right\rbrace=0.
\end{equation}
To solve the optimization part of \eqref{energy_hjb}, we denote
\begin{equation*}
F(v)=\nabla\psi^T(v+u)-\|v\|^2-C
\end{equation*}
and calculate its critical point as
\begin{equation*}
v^*=\frac{1}{2}\nabla\psi,
\end{equation*}
which means that the optimal is
\begin{equation}\label{energy_hamiltonian}
H=\sup_{v:\|v\|\leq V}F(v)=\frac{1}{4}\|\nabla\psi\|^2+\nabla\psi^Tu-C
\end{equation}
and by \eqref{energy_psi_set2}, we have 
\begin{equation}
\begin{aligned}
\|v^*\|^2&=\frac{1}{4}\|\nabla\psi\|^2=\frac{\|x-x_0\|^2}{t^2}+\|u\|^2-\frac{2(x-x_0)^Tu}{t} \\
& \leq V^2,
\end{aligned}
\end{equation}
which leads to the fact that $F(v^*)=\sup_{v:\|v\|\leq V}F(v)$. Let us take (\ref{energy_psi_set2}),(\ref{energy_psi_set3}) into \eqref{energy_hamiltonian} and the result is
\begin{equation}\label{energy_H}
H=\frac{\|x-x_0\|^2}{t^2}-\|u\|^2-C.
\end{equation}
Combining (\ref{energy_psi_set1}) and \eqref{energy_H} finally results in the constructed $\psi$ being the solution of \eqref{energy_hjb}.

Based on the solution $\psi$, we further find the minimizer over time $t$ and solve the minimization problem as follow:
\begin{equation*}
\begin{aligned}
&\min_{t\geq0}\psi=\min_{t\geq0}(\|v\|^2+C)t\\=&\frac{\|x-x_0\|^2}{t}-2(x-x_0)^Tu+(C+\|u\|^2)t.
\end{aligned}
\end{equation*}
It is easy to see that the global minimizer of $\psi$ over $t$ is
\begin{equation*}
t^*=\frac{\|x-x_0\|}{\sqrt{C+\|u\|^2}}
\end{equation*}
and the corresponding minimum is
\begin{equation}\label{energy_critical_pt}
\psi^*=2\|x-x_0\|\sqrt{C+\|u\|^2}-2(x-x_0)^Tu.
\end{equation}
Thus, if $t^*$ is reachable, that is, using \eqref{energy_assumption}, we have
\begin{equation*}
\|v(t^*)\|^2=C+2\|u\|^2-\sqrt{C+\|u\|^2}\frac{2(x-x_0)^Tu}{\|x-x_0\|}\leq V^2
\end{equation*}
the optimal is given as \eqref{energy_critical_pt}.

On the other hand, if $\|v(t^*)\|>V$, the global minimizer $t^*$ is on longer in the domain of our problem. In this case, we notice that $\|v\|^2$ is decreasing on the interval
\begin{equation*}
\left[t^*,\frac{(x-x_0)^Tu}{\|x-x_0\|^2}\right],
\end{equation*}
and is increasing on
\begin{equation*}
\left[\frac{(x-x_0)^Tu}{\|x-x_0\|^2},+\infty\right).
\end{equation*}
Also by noticing that $\lim_{t\rightarrow+\infty}\|v\|^2=\|u\|^2\leq V^2$, we conclude that there exists $t_0>t^*$ when $t\geq t_0>t^*$, $\|v\|\leq V$. Meanwhile, when $t>t^*$, $\psi$ is monotone increasing with respect to $t$. Hence, to get the minimum, we should take the time $t=t_0$, where $\|v(t_0)\|=V$. By taking the equality in \eqref{energy_assumption}, we have then
\begin{equation*}
(\|u\|^2-V^2)t^2-2(x-x_0)^Tut+\|x-x_0\|^2=0,
\end{equation*}
from which we have
\begin{equation*}
\begin{aligned}
t_0=&\frac{1}{\|u\|^2-V^2}\Big((x-x_0)^Tu\\&-\sqrt{\left((x-x_0)^Tu\right)^2+\|x-x_0\|^2(V^2-\|u\|^2)}\Big),
\end{aligned}
\end{equation*}
and $\min_{t}\psi=(V^2+C)t_0$. 
\end{proof}

\begin{proof}[Proof of Theorem \ref{thm:general}]
Denoting $g(w)$ to be the inverse of $f$ such that if $w=f(u+v)$ then $u+v=g(w)$, we will show that
\begin{equation*}
\psi(x,t)=\left\lbrace\begin{array}{ll}
tL\left(g\left(\frac{x-x_0}{t}\right)-u\right)&\|g\left(\frac{x-x_0}{t}\right)-u\|\leq V\\
+\infty&\text{otherwise}
\end{array}\right.
\end{equation*}
satisfies the HJB equation \eqref{HJB}. First of all, the Hessian matrix $\mathcal{H}(v)$ is positive definite for all $\|v\|\leq V$ since $L$ is convex. Therefore, for any $v_1,v_2$ in the domain, there exists $\xi$ such that 
\begin{equation*}
\nabla_vL(v_1)=\nabla_vL(v_2)+\mathcal{H}(\xi)(v_2-v_1).
\end{equation*}
Further if $\nabla_vL(v_1)=\nabla_vL(v_2)$, then $\mathcal{H}(\xi)(v_2-v_1)=0$. Because of the positive definite property for $\mathcal{H}$, we have $v_2=v_1$, which implies that $\nabla_vL(v)$ is one-to-one.

Then we do the following calculation on the non-infinity part of $\psi$
\begin{align}
&\begin{aligned}
&\psi_t=L\left(g\left(\frac{x-x_0}{t}\right)-u\right)\\&-\left[\nabla_vL\left(g\left(\frac{x-x_0}{t}\right)-u\right)\right]^T\nabla_wg\left(\frac{x-x_0}{t}\right)\frac{x-x_0}{t}\label{equ:general_psi_t},
\end{aligned}\\
&\nabla\psi=\left[\nabla_wg\left(\frac{x-x_0}{t}\right)\right]^T\nabla_vL\left(g\left(\frac{x-x_0}{t}\right)-u\right).\label{equ:general_psi_x}
\end{align}
Since $L$ is convex, we further have the relaxed optimization
\begin{equation*}
\max_v\{\nabla\psi^T(v+u)-L(v)\}
\end{equation*}
is a convex problem and get the condition for the optimal $v^*$ to be 
\begin{equation*}
\begin{aligned}
&\left[\nabla_vf(u+v^*)\right]^T\nabla\psi=\nabla_vL(v^*)\\
\Longrightarrow&\nabla\psi=\left[\nabla_wg(f(u+v^*))\right]^T\nabla_vL(v^*).
\end{aligned}
\end{equation*}
Combining this with \eqref{equ:general_psi_x}, we have
\begin{equation}\label{equ:general_opt_v}
v^*=g\left(\frac{x-x_0}{t}\right)-u,
\end{equation}
and $\|v^*\|\leq V$ holds. Thus, $v^*$ is the maximizer of $H(x,\nabla\psi)$. Taking \eqref{equ:general_psi_t}, \eqref{equ:general_psi_x} and \eqref{equ:general_opt_v}, we have
\begin{equation*}
\psi_t+\nabla\psi^T(v^*+u)-L(v^*)=0,
\end{equation*}
implying that \eqref{HJB} holds. At last notice that
\[
\lim_{t\rightarrow\infty}L\left(g\left(\frac{x-x_0}{t}\right)-u\right)=L(g(0)-u)<\infty,
\]
since $0$ is in the range of $f$. We have that 
\[
\lim_{t\rightarrow\infty}tL\left(g\left(\frac{x-x_0}{t}\right)-u\right)=\infty.
\]
Thus, we have $t^*>0$ such that given $x$,
\begin{equation*}
t^*=\argmin_{t\geq0}\psi(x,t).
\end{equation*}
Thus, $v^*=g((x-x_0)/t^*)-u$ gives us a constant velocity motion.
\end{proof}





\bibliography{main.bib}

\newcommand{\noop}[1]{}
\begin{thebibliography}{30}
\providecommand{\natexlab}[1]{#1}
\providecommand{\url}[1]{{#1}}
\providecommand{\urlprefix}{URL }
\providecommand{\doi}[1]{\url{https://doi.org/#1}}
\providecommand{\eprint}[2][]{\url{#2}}
 \bibcommenthead

\bibitem[{Chen et~al(2019)Chen, He, and Li}]{chen2019horizon}
Chen Y, He Z, Li S (2019) Horizon-based lazy optimal rrt for fast, efficient
  replanning in dynamic environment. Autonomous Robots 43(8):2271--2292

\bibitem[{Chow et~al(2013)Chow, Yang, and Zhou}]{chow2013global}
Chow SN, Yang TS, Zhou HM (2013) Global optimizations by intermittent
  diffusion. In: Chaos, CNN, Memristors and Beyond: A Festschrift for Leon Chua
  With DVD-ROM, composed by Eleonora Bilotta. World Scientific, p 466--479

\bibitem[{Forrest et~al(1974)Forrest, Hirst, and Tomlin}]{forrest1974practical}
Forrest J, Hirst J, Tomlin JA (1974) Practical solution of large mixed integer
  programming problems with umpire. Management Science 20(5):736--773

\bibitem[{Gammell et~al(2018)Gammell, Barfoot, and
  Srinivasa}]{gammell2018informed}
Gammell JD, Barfoot TD, Srinivasa SS (2018) Informed sampling for
  asymptotically optimal path planning. IEEE Transactions on Robotics
  34(4):966--984

\bibitem[{Hou et~al(2019)Hou, Zhai, Zhou, and Zhang}]{hou2019partitioning}
Hou M, Zhai H, Zhou H, et~al (2019) Partitioning ocean flow field for
  underwater vehicle path planning. In: OCEANS 2019-Marseille, IEEE, pp 1--8

\bibitem[{Hou et~al(2021)Hou, Cho, Zhou, Edwards, and
  Zhang}]{10.3389/frobt.2021.575267}
Hou M, Cho S, Zhou H, et~al (2021) Bounded cost path planning for underwater
  vehicles assisted by a time-invariant partitioned flow field model. Frontiers
  in Robotics and AI 8:203. \doi{10.3389/frobt.2021.575267},
  \urlprefix\url{https://www.frontiersin.org/article/10.3389/frobt.2021.575267}

\bibitem[{Janson et~al(2015)Janson, Schmerling, Clark, and
  Pavone}]{janson2015fast}
Janson L, Schmerling E, Clark A, et~al (2015) Fast marching tree: A fast
  marching sampling-based method for optimal motion planning in many
  dimensions. The International journal of robotics research 34(7):883--921

\bibitem[{Ji et~al(2019)Ji, Choi, Kang, Cho, Joo, and Lee}]{ji2019design}
Ji DH, Choi HS, Kang JI, et~al (2019) Design and control of hybrid underwater
  glider. Advances in Mechanical Engineering 11(5):1687814019848,556

\bibitem[{Kaiser et~al(2014)Kaiser, Noack, Cordier, Spohn, Segond, Abel,
  Daviller, {\"O}sth, Krajnovi{\'c}, and Niven}]{kaiser2014cluster}
Kaiser E, Noack BR, Cordier L, et~al (2014) Cluster-based reduced-order
  modelling of a mixing layer. Journal of Fluid Mechanics 754:365--414

\bibitem[{Karaman and Frazzoli(2011)}]{karaman2011sampling}
Karaman S, Frazzoli E (2011) Sampling-based algorithms for optimal motion
  planning. The international journal of robotics research 30(7):846--894

\bibitem[{Kuffner and La{V}alle(2000)}]{Kuffner2000}
Kuffner J, La{V}alle S (2000) {R}{R}{T}-connect: An efficient approach to
  single-query path planning. In: IEEE International Conference on Robotics and
  Automation

\bibitem[{Kularatne et~al(2017)Kularatne, Bhattacharya, and
  Hsieh}]{kularatne2017optimal}
Kularatne D, Bhattacharya S, Hsieh MA (2017) Optimal path planning in
  time-varying flows using adaptive discretization. IEEE Robotics and
  Automation Letters 3(1):458--465

\bibitem[{Kularatne et~al(2018)Kularatne, Bhattacharya, and
  Hsieh}]{kularatne2018going}
Kularatne D, Bhattacharya S, Hsieh MA (2018) Going with the flow: a graph based
  approach to optimal path planning in general flows. Autonomous Robots
  42(7):1369--1387

\bibitem[{La{V}alle(1998)}]{Lavalle1998}
La{V}alle SM (1998) {R}apidly-{E}xploring {R}andom {T}rees: A new tool for path
  planning. Tech. rep., Department of Computer Science, Iowa State University

\bibitem[{Leonard et~al(2010)Leonard, Paley, Davis, Fratantoni, Lekien, and
  Zhang}]{Leonard2010a}
Leonard NE, Paley DA, Davis RE, et~al (2010) {Coordinated control of an
  underwater glider fleet in an adaptive ocean sampling field experiment in
  Monterey Bay}. Journal of Field Robotics 27(6):718--740.
  \doi{10.1002/rob.20366},
  \urlprefix\url{http://doi.wiley.com/10.1002/rob.20366}

\bibitem[{Li et~al(2017)Li, Lu, Zhou, and Chow}]{li2017method}
Li W, Lu J, Zhou H, et~al (2017) Method of evolving junctions: A new approach
  to optimal control with constraints. Automatica 78:72--78

\bibitem[{Lolla(2016)}]{Lolla2016}
Lolla SVT (2016) Path planning and adaptive sampling in the coastal ocean. PhD
  thesis, Massachusetts Institute of Technology

\bibitem[{Martin(2000)}]{martin2000}
Martin PJ (2000) Description of the navy coastal ocean model version 1.0. Tech.
  Rep. NRL/FR/7322--00-9962, Naval Research Lab

\bibitem[{Noreen et~al(2016)Noreen, Khan, and Habib}]{noreen2016optimal}
Noreen I, Khan A, Habib Z (2016) Optimal path planning using {RRT}* based
  approaches: a survey and future directions. International Journal of Advanced
  Computer Science and Applications 7(11):97--107

\bibitem[{Ozog et~al(2016)Ozog, Carlevaris-Bianco, Kim, and
  Eustice}]{Ozog2016LongtermMT}
Ozog P, Carlevaris-Bianco N, Kim AY, et~al (2016) Long-term mapping techniques
  for ship hull inspection and surveillance using an autonomous underwater
  vehicle. Journal of Field Robotics 33:265--289

\bibitem[{Pereira et~al(2013)Pereira, Binney, Hollinger, and
  Sukhatme}]{Pereira2013}
Pereira AA, Binney J, Hollinger GA, et~al (2013) Risk-aware path planning for
  autonomous underwater vehicles using predictive ocean models. Journal of
  Field Robotics 30(5):741--762. \doi{10.1002/rob.21472}

\bibitem[{Poole and Mackworth(2010)}]{poole2010artificial}
Poole DL, Mackworth AK (2010) Artificial Intelligence: foundations of
  computational agents. Cambridge University Press

\bibitem[{Rhoads et~al(2012)Rhoads, Mezic, and Poje}]{Rhoads2012}
Rhoads B, Mezic I, Poje AC (2012) Minimum time heading control of underpowered
  vehicles in time-varying ocean currents. Ocean Engineering 66(1):12--31

\bibitem[{Russell and Norvig(2002)}]{russell2002artificial}
Russell S, Norvig P (2002) Artificial intelligence: a modern approach

\bibitem[{Ser-Giacomi et~al(2015)Ser-Giacomi, Rossi, L{\'o}pez, and
  Hern{\'a}ndez-Garc{\'\i}a}]{ser2015flow}
Ser-Giacomi E, Rossi V, L{\'o}pez C, et~al (2015) Flow networks: A
  characterization of geophysical fluid transport. Chaos: An Interdisciplinary
  Journal of Nonlinear Science 25(3):036,404

\bibitem[{Sethian(1999)}]{Sethian1999}
Sethian JA (1999) Level Set Methods and Fast Marching Methods: Evolving
  Interfaces in Geometry, Fluid Mechanics, Computer Vison and Material Science.
  Cambridge University Press

\bibitem[{Shome et~al(2020)Shome, Solovey, Dobson, Halperin, and
  Bekris}]{shome2020drrt}
Shome R, Solovey K, Dobson A, et~al (2020) {dRRT}*: Scalable and informed
  asymptotically-optimal multi-robot motion planning. Autonomous Robots
  44(3):443--467

\bibitem[{Sinha et~al(2017)Sinha, Malo, and Deb}]{sinha2017review}
Sinha A, Malo P, Deb K (2017) A review on bilevel optimization: from classical
  to evolutionary approaches and applications. IEEE Transactions on
  Evolutionary Computation 22(2):276--295

\bibitem[{Smith et~al(2010)Smith, Chao, Li, Caron, Jones, and
  Sukhatme}]{Smith2010}
Smith RN, Chao Y, Li PP, et~al (2010) Planning and implementing trajectories
  for autonomous underwater vehicles to track evolving ocean processes based on
  predictions from a {R}egional {O}cean {M}odel. The International Journal of
  Robotics Research 29(12):1475--1497

\bibitem[{Soulignac(2011)}]{Soulignac2011}
Soulignac M (2011) Feasible and optimal path planning in strong current fields.
  {IEEE} Transactions on Robotics 27(1):89--98. \doi{10.1109/tro.2010.2085790}

\end{thebibliography}

\end{document}